\documentclass{scrartcl}

\usepackage{amssymb}
\usepackage{amsmath}
\usepackage{hyperref}
\usepackage{xcolor}
\usepackage{pgf}
\usepackage[left=2.5cm,right=2.5cm,top=1.5cm,bottom=1.5cm,includeheadfoot]{geometry}
\usepackage[singlespacing]{setspace}

\newtheorem{theorem}{Theorem}[section]
\newtheorem{lemma}[theorem]{Lemma}
\newtheorem{corollary}[theorem]{Corollary}

\newtheorem{definition}[theorem]{Definition}
\newtheorem{example}[theorem]{Example}

\newtheorem{remark}[theorem]{Remark}

\newenvironment{proof}{\paragraph{Proof:}}{\hfill$\square$}
\let\origitemize\itemize
\def\itemize{\origitemize\itemsep0pt}
\numberwithin{equation}{section}

\makeatletter
\DeclareOldFontCommand{\rm}{\normalfont\rmfamily}{\mathrm}
\DeclareOldFontCommand{\sf}{\normalfont\sffamily}{\mathsf}
\DeclareOldFontCommand{\tt}{\normalfont\ttfamily}{\mathtt}
\DeclareOldFontCommand{\bf}{\normalfont\bfseries}{\mathbf}
\DeclareOldFontCommand{\it}{\normalfont\itshape}{\mathit}
\DeclareOldFontCommand{\sl}{\normalfont\slshape}{\@nomath\sl}
\DeclareOldFontCommand{\sc}{\normalfont\scshape}{\@nomath\sc}
\makeatother
\addtokomafont{disposition}{\rmfamily}
\begin{document}

\title{
	Nonlinear Cone Separation Theorems in Real Topological Linear Spaces
}

\author{Christian Günther\thanks{Leibniz Universität Hannover, Institut für Angewandte Mathematik, Welfengarten 1, 30167 
		Hannover, Germany
		(\href{mailto:c.guenther@ifam.uni-hannover.de}{c.guenther@ifam.uni-hannover.de}).}
	\and Bahareh Khazayel\thanks{Martin Luther University Halle-Wittenberg, Faculty
		of Natural Sciences II, Institute of Mathematics, 06099 Halle (Saale), Germany
		(\href{mailto:Bahareh.Khazayel@mathematik.uni-halle.de}{Bahareh.Khazayel@mathematik.uni-halle.de}).}
	\and Christiane Tammer\thanks{Martin Luther University Halle-Wittenberg, Faculty
		of Natural Sciences II, Institute of Mathematics, 06099 Halle (Saale), Germany
		(\href{mailto:Christiane.Tammer@mathematik.uni-halle.de}{Christiane.Tammer@mathematik.uni-halle.de}).}}

\maketitle

\begin{center}
	\textbf{Abstract}
\end{center}
\begin{abstract}
	The separation of two sets (or more specific of two cones) plays an important role in different fields of mathematics such as variational analysis, convex analysis, convex geometry, optimization. In the paper, we derive some new results for the separation of two not necessarily convex cones by a (convex) cone / conical surface in real (topological) linear spaces. Basically, we follow the separation approach by Kasimbeyli (2010, SIAM J. Optim. 20) based on augmented dual cones and Bishop-Phelps type (normlinear) separating functions. Classical separation theorems for convex sets are the key tool for proving our main nonlinear cone separation theorems. 
\end{abstract}


\begin{flushleft}
\end{flushleft}

\begin{flushleft}
	\textbf{Mathematics subject classifications (MSC 2010):} 46A22, 49J27, 65K10, 90C48 
\end{flushleft}

\section{Introduction}

It is well-known that the separation of two sets (or more specific of two cones) plays an important role in different fields of mathematics (such as variational analysis, convex analysis, convex geometry, optimization). Basically the following question is of interest: Under which assumptions on the real (topological) linear space $E$ and the sets $\Omega^1, \Omega^2 \subseteq E$ one can ensure the existence of a nontrivial element $x^*$ of the (algebraic; topological) dual space of $E$ such that the weak [proper] separation condition
\begin{equation}\label{f11}\sup_{\omega^1 \in \Omega^1}\, x^*(\omega^1) \leq \inf_{\omega^2 \in \Omega^2}\, x^*(\omega^2) \quad \left[\;\mbox{and} \quad \inf_{\omega^1 \in \Omega^1}\, x^*(\omega^1)  < \sup_{\omega^2 \in \Omega^2}\, x^*(\omega^2)\;\right],\end{equation}
or the strict separation condition
$$\sup_{\omega^1 \in \Omega^1}\, x^*(\omega^1)  < \inf_{\omega^2 \in \Omega^2}\, x^*(\omega^2)$$
holds true? For proving the existence of such a $x^*$ one needs usually some convexity assumptions, certain closedness / compactness / solidness assumptions and a disjointedness assumption on $\Omega^1$ and $\Omega^2$. 
From the historical point of view, assertions about the existence of linear separating functions for two convex sets such that \eqref{f11} holds (also known as Hahn-Banach separation theorems or hyperplane separation theorems) were developed from the classical Hahn-Banach theory (i.e., theorems related to the extension of bounded linear functions defined on a subspace of some linear space to the whole space) which is a central tool in functional analysis. Geometric versions of the Hahn–Banach theorem go back to works by Dieudonne, by Eidelheit, by Holmes, by James, by Mazur, and by Tukey. To get a further overview on this topic, we refer the reader also to the references by Aliprantis and Border \cite[Ch. 5]{AliprantisBorder},  Deumlich, Elster and Nehse \cite{DeumlichElsterNehse}, Elster and Nehse \cite{ElsterNehse1978},  Göpfert {\rm et al.} \cite[Sec. 2.2]{GoeRiaTamZal2003}, Holmes \cite{Holmes}, Jahn \cite[Ch. 3]{Jahn2011}, Khan, Tammer and Z\u{a}linescu \cite[Ch. 5 and 6]{Khanetal2015}, Soltan \cite{Soltan2021},  Z\u{a}linescu \cite[Ch. 1]{Zali2002}, and references therein.
 
Removing the convexity assumptions on $\Omega^1$ and $\Omega^2$, it is easy to observe that one cannot expect the existence of such a linear function $x^*: E \to \mathbb{R}$ in the separation condition \eqref{f11}. Thus, the idea arises to replace the linear function $x^*$ in the separation condition \eqref{f11} by a nonlinear function $\varphi: E \to \mathbb{R}$. Kasimbeyli \cite{Kasimbeyli2010} (compare also \cite{AcaKas21}, \cite{BagAdi13}, \cite{GasOzt06}, \cite{OztGur15}) had the idea to use (in a reflexive Banach space setting) a Bishop-Phelps type (normlinear) function $\varphi_{x^*, \alpha}: E \to \mathbb{R}$ (where some linear continuous function $x^*$ in the topological dual space $E^*$ of $E$, the norm $\mid \mid \cdot \mid \mid $ on $E$, and some $\alpha \geq 0$ are involved), which is defined by
\begin{equation}\label{f12}
\varphi_{x^*, \alpha}(x)  := x^*(x) + \alpha ||x|| \quad \mbox{for all } x \in E,
\end{equation}
as a nonlinear separating function. The type of function $\varphi_{x^*, \alpha}$ in \eqref{f12} is associated to so-called Bishop-Phelps cones and dates back to the work by Bishop and Phelps \cite{BP1962} (see Section \ref{sec:bp_cones} for more details). Observe, the combination of a norm function with a linear (continuous) function in \eqref{f12} leads also to the name ``normlinear function'' (see the recent paper by Zaffaroni \cite{Zaffaroni2022} with $x^* \in E^*$ and $\alpha \in \mathbb{R}$ in \eqref{f12}). 
Clearly, also other nonlinear functions, for instance by Drummond and Svaiter \cite{DrummondSvaiter}, Gerstewitz \cite{Gerstewitz1983} or Hiriart-Urruty \cite{HiriartUrruty} (see also Bouza, Quintana and Tammer \cite{BouzaQuintanaTammer}, Gerth and Weidner \cite{GerthWeidner}, Göpfert {\rm et al.} \cite[Sec. 2.3]{GoeRiaTamZal2003}, Ha \cite[Sec. 3.1 and 3.2]{Ha2022}, Jahn \cite{Jahn2022, Jahn2023}, Tammer and Weidner \cite[Sec. 4.11]{TammerWeidner2020}, Zaffaroni \cite[Prop. 3.2]{Zaffaroni2003}), could be used in a nonlinear separation approach for not necessarily convex sets.

Furthermore, it is important to mention that the extremal principle by Kruger and Mordukhovich \cite{KruMor80a,Kru98,Mor76,Mor94} is a dual approach to separation of not necessarily convex sets. This extremal principle is a deep assertion concerning the existence of linear continuous functions with important separation properties concerning nonconvex sets and essential applications in variational analysis, geometrical approaches  and optimization. Recent developments and extensions of the extremal principle are given by Bui and Kruger in \cite{BuiKru2018}, \cite{BuiKru2019}.

Consider an extremal system $\{ \Omega^1, \ldots ,\Omega^n , \overline{x} \}$ (see Mordukhovich \cite{Mordukhovich2006a} and Mordukhovich, Nam \cite[Definition 3.6]{MorNam2022})  in a real normed space $E$. 
Then, $\{ \Omega^1, \ldots ,\Omega^n , \overline{x} \}$ fulfills the {exact extremal principle} if there are basic normals
\begin{equation}\label{f-extr4}
x^*_j \in N_M (\overline{x}; \Omega^j ), \quad j = 1, \ldots ,n,
\end{equation}
such that
\begin{equation}\label{f-extr2}
x^*_1 + \cdots + x^*_n = 0, \qquad \|x_1^* \|_* + \cdots + \|x^*_n\|_* =1 
\end{equation}
holds. Here, $N_M (\overline{x}; \Omega^j )$ denotes the Mordukhovich (limiting) normal cone (see \cite[Definition 7.66]{MorNam2022}), and $\|\cdot\|_*$ denotes the dual norm in $E^*$.

Considering the exact extremal principle for two sets $\Omega^1$ and $\Omega^2$, formula (\ref{f-extr4}) and (\ref{f-extr2}) reduce to the condition that there is an element $x^* \in E^* \setminus \{0\}$ with
\begin{equation}\label{l-sep-extrpr}
 x^* \in N_M (\overline{x}; \Omega^1 ) \cap (- N_M (\overline{x}; \Omega^2 )).
\end{equation}
If $\Omega^1$ and $\Omega^2$ are convex sets and $\Omega^1 - \Omega^2 $ is topologically solid, then (\ref{l-sep-extrpr}) coincides with
\[\forall \; x_1 \in \Omega^1, \; \forall x_2 \in \Omega^2: \qquad  x^*( x_1 ) \leq  x^*( x_2) , \]
i.e., we get the classical weak separation property (compare \eqref{f11}) for two convex sets, see Mordukhovich, Nam \cite[Theorem 3.7]{MorNam2022}.

\medskip

In our paper, the separating function $\varphi: E \to \mathbb{R}$ will be given by a combination of a seminorm $\psi: E \to \mathbb{R}$ with a linear function $x^*$, i.e.,
$\varphi := \varphi_{x^*, \alpha}: E \to \mathbb{R}$ (for some $x^*$ in the  (algebraic; topological) dual space of $E$, and some $\alpha \geq 0$), where
\begin{equation}
    \label{f16}
    \varphi_{x^*, \alpha}(x)  := x^*(x) + \alpha \psi(x) \quad \mbox{for all } x \in E.
\end{equation}
We will use functions of type \eqref{f16} in order to derive separation conditions for not necessarily convex cones within the framework of real topological linear spaces having in mind certain applications in vector optimization and order theory (see \cite{Kasimbeyli2010}).

Similar to the work by Kasimbeyli \cite{Kasimbeyli2010}, we are primarily interested in the situation that both sets $\Omega^1$ and $\Omega^2$ of the real topological linear space $E$ (with underlying topology $\tau$) are given by cones (i.e., $\emptyset \neq \Omega^i = \mathbb{R}_+ \cdot \Omega^i$, $i = 1, 2$). In the literature, cone separation theorems are derived in different ways (see, e.g., Henig \cite{Henig}, Nehse \cite{Nehse1981}, Nieuwenhuis \cite{Nieuwenhuis}). We adapt the ideas by Henig \cite{Henig} and by Kasimbeyli \cite{Kasimbeyli2010} to separate two (not necessarily convex) cones by a cone (conical surface). Consider two nontrivial cones $\Omega^1, \Omega^2 \subseteq E$ and a (closed, convex) cone $C \subseteq E$ with nonempty interior (denoted by ${\rm int}_\tau\, C$; furthermore, ${\rm bd}_\tau\, C$ denotes the boundary,  ${\rm cl}_\tau\, C$ denotes the closure of $C$). We say that the cones $\Omega^1$ and $\Omega^2$ are \\
\begin{itemize}
    \item[(i)] {\bf (weakly) separated by (the boundary of) the cone $C$} if 
    \begin{equation}
    \label{eq:weak_separation}
     \Omega^1 \subseteq {\rm cl}_\tau\, C \quad \mbox{and} \quad \Omega^2 \subseteq E \setminus {\rm int}_\tau\, C.   
    \end{equation}
    \item[(ii)] {\bf properly separated by (the boundary of) $C$} if \eqref{eq:weak_separation} is valid and
    $$ (\Omega^1 \cup \Omega^2) \setminus {\rm bd}_\tau\, C \neq \emptyset.$$
    \item[(iii)] {\bf strictly separated by (the boundary of)  $C$} if 
    $$
    \Omega^1 \setminus \{0\} \subseteq {\rm int}_\tau\, C \quad \mbox{and} \quad \Omega^2 \setminus \{0\} \subseteq E \setminus {\rm cl}_\tau\, C.
    $$
\end{itemize}
Notice that (iii) $\Longrightarrow$ (ii) $\Longrightarrow$ (i).
In particular, in our {\bf (Weak, Proper, Strict) Cone Separation Theorems} we like to consider a cone $C$ that can be written in term of a zero lower-level set of a function $\varphi: E \to \mathbb{R}$ that is (lower semi-continuous, sublinear, convex) positive homogeneous. More precisely, we like to ensure the cone representation properties 
$$
{\rm cl}_\tau\, C = \{x \in E \mid \varphi(x) \leq 0 \} \quad \mbox{and} \quad {\rm int}_\tau\, C = \{x \in E \mid \varphi(x) < 0 \}.
$$
In our paper, the function $\varphi$ will be given by the function $\varphi_{x^*, \alpha}$ defined in \eqref{f16}.
Notice that the above cone separation approach also covers the linear cone separation case where the cone $C$ is given by a closed halfspace, i.e.,
$$
C = {\rm cl}_\tau\, C = \{x \in E \mid x^*(x) \leq 0 \} \quad \mbox{and} \quad {\rm int}_\tau\, C = \{x \in E \mid x^*(x) < 0 \}.
$$
In contrast to linear cone separation where the separating object is a hyperplane ${\rm bd}_\tau\, C = \{x \in E \mid x^*(x) = 0 \}$, in nonlinear cone separation we will have a separating object given by a conical surface ${\rm bd}_\tau\, C = \{x \in E \mid \varphi(x) = 0 \}$.

 Cone separation theorems are studied by several authors in the literature, among others by Adan and Novo \cite[Th. 2.1]{AdanNovo2004}, Göpfert {\rm et al.} \cite[Th. 2.3.6]{GoeRiaTamZal2003}, Henig \cite{Henig}, Jahn \cite[Th. 3.22]{Jahn2011}, \cite[Sec. 3.7]{Jahn2023}, Kasimbeyli \cite[Th. 4.3]{Kasimbeyli2010}, Khazayel {\rm et al.} \cite[Sec. 2.3]{Khazayel2021a}, Nieuwenhuis \cite{Nieuwenhuis}, and Novo and Z\u{a}linescu \cite[Cor. 2.3]{NovoZali2021}.
Such cone separation theorems are known to be important for some fields of optimization (for instance, for deriving scalarization results for nonconvex vector optimization problems; see, e.g., {Bo{\c{t}}}, Grad and Wanka \cite{BotGradWanka}, Gerth and Weidner \cite{GerthWeidner}, Eichfelder and Kasimbeyli \cite{EichfelderKasimbeyli2014}, Göpfert {\rm et al.} \cite{GoeRiaTamZal2003}, Grad \cite{Grad2015}, Kasimbeyli \cite{Kasimbeyli2010, Kasimbeyli2013}, Kasimbeyli \cite{NKasimbeyli15, NKasimbeyli19}, Kasimbeyli {\rm et al.} \cite{Kasimbeyli2019}, and  Tammer and Weidner \cite{TammerWeidner2020}). In particular, to gain new results in the nonconvex setting (since some well-known cones are nonconvex cones) is of special interest. Clearly, also in a next step one may derive from cone separation theorems also (weak, proper and strict) separation theorems for nonconvex sets as well as duality assertions (see, e.g., Kasimbeyli and Karimi \cite{KasimbeyliKarimi2019, KasimbeyliKarimi21}).

\medskip

The paper is structured as follows. In the second section, we present preliminaries in real (topological) linear spaces. We will deal with certain topologies defined on real linear spaces; algebraic and topological notions for sets; convex sets, norms and seminorms; cones and their bases.
In the third and fourth sections, we introduce a new class of nonlinear separating functions based on a combination of a linear function and a seminorm. Generalized augmented dual cones and their properties are analyzed, and some relationships to Bishops-Phelps cones are pointed out. We are able to extend the framework used in the separation approach by Kasimbeyli \cite{Kasimbeyli2010}. 
The fifth section contains the main results of the paper, namely nonlinear cone separation results in real linear spaces, in real topological linear spaces and in real locally convex spaces, respectively. 
We end the paper with some concluding remarks in the sixth section.

\section{Preliminaries}

\subsection{Basics in Topological Linear Spaces}	
Throughout the paper, let $E \neq \{0\}$ be a real linear space, and let $E'$ be its algebraic dual space, which is given by
$$
E' = \{x': E \to \mathbb{R} \mid x' \mbox{ is linear}\}.
$$
Having a real topological linear space $(E, \tau)$ with topology $\tau$, we also consider the topological dual space of $E$, which is given by
$$
(E, \tau)^* = \{x^*: E \to \mathbb{R} \mid x^* \mbox{ is linear and } \tau \mbox{-continuous}\}.
$$
For simplicity we will write $E^*$ instead of $(E, \tau)^*$ when the underlying topology $\tau$ is clear in the context. 

\begin{remark}[Convex core topology] \label{rem:cctopology}
    It is well-known that any linear space $E$ can be endowed with
    the strongest locally convex topology $\tau_c$, that is generated by the family of all the
    semi-norms defined on $E$ (see Khan, Tammer and Z\u{a}linescu \cite[Sec. 6.3]{Khanetal2015}). In the literature, the topology $\tau_c$ is known as the convex core topology. According to \cite[Prop. 6.3.1]{Khanetal2015}, the topological dual space of $E$, namely 
    $(E, \tau_c)^*$, is exactly the algebraic dual space $E'$. 
    In recent works (see, e.g., Günther, Khazayel and Tammer \cite{Khazayel2021b, Khazayel2022}, Khazayel {\rm et al.} \cite{Khazayel2021a}, Novo and Z\u{a}linescu \cite{NovoZali2021}), the convex core topology $\tau_c$ is used to derive properties for algebraic interiority notions (such as core and intrinsic core).
\end{remark}

Throughout the paper, $\mathbb{R}_+$ denotes the set of nonnegative real numbers, while $\mathbb{P} := \mathbb{R}_{++}$ denotes the set of positive real numbers. 
For any two points $x, \overline{x} \in E$, the closed, the open, the half-open line segments are defined by
\begin{align*}
[x,\overline{x}] & := \{ (1-\lambda) x + \lambda \overline{x} \mid \lambda \in [0,1]\}, & 
(x,\overline{x}) & := \{ (1-\lambda) x + \lambda \overline{x} \mid \lambda \in (0,1)\},\\
[x,\overline{x}) & := \{ (1-\lambda) x + \lambda \overline{x} \mid \lambda \in [0,1)\}, & 
(x,\overline{x}] & := \{ (1-\lambda) x + \lambda \overline{x} \mid \lambda \in (0,1]\}.
\end{align*}
Consider any set $\Omega \subseteq  E$ in the linear space $E$. The smallest affine subspace of $E$ containing $\Omega$ is denoted by ${\rm aff}\,\Omega$. As usual, the algebraic interior (or the core) of $\Omega$ is defined by
	$$
	{\rm cor}\, \Omega := \{x \in \Omega \mid \forall\, v \in E\; \exists\, \varepsilon > 0: \; x + [0, \varepsilon] \cdot v \subseteq \Omega\},
	$$
while the relative algebraic interior (the intrinsic core) of $\Omega$ is given by
	$$
	{\rm icor}\, \Omega := \{x \in \Omega \mid \forall\, v \in {\rm aff}(\Omega - \Omega)\; \exists\, \varepsilon > 0: \; x + [0, \varepsilon] \cdot v \subseteq \Omega\}.
	$$
Notice, for any nonempty  
set $\Omega \subseteq E$, we have ${\rm  cor}\, \Omega = {\rm  icor}\, \Omega$ if ${\rm aff}\,\Omega = E$, and ${\rm  cor}\, \Omega = \emptyset$ otherwise. If ${\rm  icor}\, \Omega \neq \emptyset$, then
$
{\rm cor}\, \Omega \neq \emptyset \iff {\rm aff}\,\Omega = E.
$
The algebraic closure of $\Omega$ consists of all linearly accessible points of $\Omega$ and is denoted by
$$
{\rm acl}\, \Omega := \{x \in E \mid \exists\, \overline{x} \in \Omega:\; [\overline{x}, x) \subseteq \Omega\}.
$$

Given a real linear topological space $(E, \tau)$ and a set $\Omega \subseteq E$, we denote by ${\rm cl}_{\tau}\,\Omega$, ${\rm bd}_{\tau}\,\Omega$, ${\rm int}_{\tau}\,\Omega$ and ${\rm rint}_{\tau}\,\Omega$ the closure, the boundary, the interior and the relative interior of $\Omega$ w.r.t. the topology $\tau$. It is known that
$\Omega  \subseteq {\rm acl}\, \Omega \subseteq {\rm cl}_{\tau}\, \Omega = ({\rm int}_{\tau}\, \Omega) \cup {\rm bd}_{\tau}\, \Omega$ and
$\Omega  \supseteq {\rm cor}\, \Omega \supseteq {\rm int}_{\tau}\, \Omega$ as well as $\Omega \supseteq {\rm icor}\, \Omega  \supseteq {\rm rint}_{\tau}\, \Omega$.

Recall that a set $\Omega \subseteq E$ is called 
\begin{itemize}
\item $\tau$-closed if ${\rm cl}_{\tau}\, \Omega = \Omega$;
\item $\tau$-compact if every family of open sets (w.r.t. $\tau$) whose union includes 
$\Omega$ contains a finite number of sets whose union includes $\Omega$;
\item $\tau$-solid if ${\rm int}_{\tau}\, \Omega \neq \emptyset$;
\item algebraically closed if ${\rm acl}\, \Omega = \Omega$;
\item (algebraically) solid if ${\rm cor}\, \Omega \neq \emptyset$;
\item relatively (algebraically) solid if ${\rm icor}\, \Omega \neq \emptyset$.
\end{itemize}

It is well-known that $\tau$-closedness implies algebraic closedness. Furthermore, 
in a locally convex space $E$, $\tau$-closedness implies $\tau_c$-closedness (where $\tau_c$ is the convex core topology).

\subsection{Convex Sets, Norms and Seminorms} \label{subsec:convex_sets_norms_snorms}

As usual, a set $\Omega \subseteq E$ is said to be convex if $(x,\bar x) \subseteq \Omega$ for all $x, \bar x \in \Omega$. The smallest convex set of $E$ containing $\Omega$ (i.e., the convex hull of $\Omega$) is denoted by ${\rm conv}(\Omega)$. Notice that $\Omega$ is convex if and only if ${\rm conv}(\Omega) = \Omega$.
Consider the topological linear space $(E, \tau)$. 
If $\Omega$ is convex and ${\rm int}_\tau\, \Omega \neq \emptyset$, then ${\rm cor}\, \Omega = {\rm int}_{\tau}\, \Omega$.
According to Cuong {\rm et al.} \cite[Th. 2.9]{MordukhovichEtAl2022}, if $E$ is locally convex and $\Omega$ is convex with ${\rm rint}_\tau\, \Omega \neq \emptyset$, then ${\rm icor}\, \Omega = {\rm rint}_\tau\, \Omega.$

\begin{remark}[Convex core topology]
    It is known that, for any convex set $\Omega \subseteq E$, we have  ${\rm cor}\, \Omega = {\rm int}_{\tau_c}\, \Omega$ and ${\rm icor}\, \Omega = {\rm rint}_{\tau_c}\, \Omega$, where $\tau_c$ is the convex core topology.
	Moreover, for any relatively solid, convex set $\Omega \subseteq E$, we have 
	${\rm icor}\, \Omega = {\rm icor}( {\rm cl}_{\tau_c}\, \Omega)$ and
	${\rm acl}\, \Omega = {\rm cl}_{\tau_c}\, \Omega = {\rm cl}_{\tau_c}({\rm icor}\, \Omega)$.
\end{remark}

Special convex sets in $E$ are given by the unit balls of seminorms. 
Recall that a function $\psi: E \to \mathbb{R}$ is called a seminorm if $\psi$ is sublinear and symmetric (hence $\psi(0) = 0$). 
The unit ball of $\psi$ is then given by the set $\{x \in E \mid \psi(x) \leq 1\}$.
Moreover, recall that a function $\psi: E \to \mathbb{R}$ is called a norm if $\psi$ is a seminorm and the implication $\psi(x) = 0 \Longrightarrow x = 0$ holds true.

In the following examples, we like to point out how one can construct a (semi)norm in a linear space $E$ (see also Tammer and Weidner \cite[Sec. 2.1]{TammerWeidner2020}).

\begin{example} \label{ex:seminorm1} Assume that $E$ is a real topological linear space and $x^* \in E^*$. Then, the function $\psi_{x^*}: E \to \mathbb{R}$, defined by 
$
\psi_{x^*}(x) := |x^*(x)|
$
for all $x \in E$,
defines a seminorm on $E$. If the dimension of $E$ is greater than $1$, then $x^*(x) = 0$ has nontrivial solutions, hence $\psi$ is not a norm.
\end{example}

\begin{example} \label{ex:seminorm1b} Given a nonempty, convex, absorbing, balanced set $B \subseteq E$ in the linear space $E$, one can define a seminorm $\psi: E \to \mathbb{R}$ by using the so-called Minkowski functional, which is defined by
$
\psi(x) := \inf\{\lambda > 0 \mid x \in \lambda \cdot B\}
$
for all $x \in E$.
\end{example}

\begin{example} \label{ex:seminorm2}
Assume that $E$ is a linear space and $\psi : E \to \mathbb{R}$ is sublinear. Then
each of the following functions is a seminorm on $E$:

$\psi_{\max}: E \to \mathbb{R}$ defined by
$
\psi_{\max}(x) := \max\{\psi(x), \psi(-x)\} 
$
for all $x \in E$;

$\psi_{\sum}: E \to \mathbb{R}$ defined by
$
\psi_{\sum}(x) := \psi(x) + \psi(-x) 
$
for all $x \in E$.
\end{example}

\begin{remark} 
In Lemma \ref{lem:semnorm_max_properties} we will state some more properties of $\psi_{\max}$.
\end{remark}

We will denote by the tripel $(E, \mathcal{F}, \tau)$ a real locally convex space $E$ with underlying topology $\tau$ that is generated by a family of seminorms $\mathcal{F}$. As mentioned in Remark \ref{rem:cctopology},  $(E, \mathcal{F}, \tau_c)$ with $\mathcal{F}$ consisting of all seminorms defined on $E$, and the convex core topology $\tau_c$, is a real locally convex space.

\begin{example} \label{ex:seminorm3}
We suppose that $(E, \mathcal{F}, \tau)$ is a separated locally convex space with a topology $\tau$ that is generated by a family $\mathcal{F} := \{s^i \mid i \in I\}$ of seminorms $s^i: E \to \mathbb{R}$, $i \in I$, with $\sup\{s^i(x) \mid i \in I\} < +\infty$ for all $x \in E$ (for instance, if $|I| < \infty$). Then, $\psi_{\mathcal{F}}: E \to \mathbb{R}$, defined by 
$
\psi_{\mathcal{F}}(x) := \sup\{s^i(x) \mid i \in I\}
$
for all $x \in E$,
is a norm on $E$. Notice, for any $y \in E \setminus \{0\}$ there is $i \in I$ such that $s^i(y) > 0$ (see Swartz \cite[pp. 169-170]{Swartz}). 
\end{example}

\subsection{Cones and their Bases} \label{subsec:cones}

Recall that a cone $K \subseteq E$ (i.e., $0 \in K = \mathbb{R}_+ \cdot K$) is convex if $K$ is convex (or equivalently $K + K = K$); nontrivial if $\{0\} \neq K \neq E$;
pointed if $\ell(K) : = K \cap (-K) = \{0\}$. Notice that $\ell(K) \subseteq K \subseteq {\rm aff}\,K$, and $K$ is a linear subspace of $E$ if and only if $K = \ell(K)$ and $K$ is convex.

For any cone $K \subseteq E$, we define its (algebraic) dual cone by
$$
K^+ := \{y' \in E' \mid \forall\, k \in K:\; y'(k) \geq 0\}.
$$
Moreover, the following two sets will be of special interest,
\begin{align*}
K^\# & := \{y' \in E' \mid \forall\, k \in K \setminus \{0\}:\; y'(k) > 0\},\\
K^\& & := \{y' \in E' \mid \forall\, k \in K \setminus \ell(K):\; y'(k) > 0\}.
\end{align*}
Obviously, we have $K^\# \subseteq K^+$ and $K^\# \subseteq K^\&$. If $K \subseteq {\rm acl}(K \setminus \ell(K))$, then $K^\# \subseteq K^\& \subseteq K^+$.

In a real topological linear space $(E, \tau)$, we define the (topological) dual cone of a cone $K \subseteq E$ by
$$
K^+_\tau := K^+ \cap E^*.
$$
Moreover, we put
$$
K^\#_\tau  := K^\# \cap E^* \quad \mbox{and} \quad 
K^\&_\tau := K^\& \cap E^*.
$$
Clearly, we have $K_\tau^\# \subseteq K_\tau^+$ and $K_\tau^\# \subseteq K_\tau^\&$. If $K \subseteq {\rm cl}_\tau(K \setminus \ell(K))$, then $K_\tau^\# \subseteq K_\tau^\& \subseteq K_\tau^+$. 
All sets $K^+, K^\#, K^\&$ and $K^+_\tau, K^\#_\tau, K^\&_\tau$ are convex for any (not necessarily convex) cone $K \subseteq E$.
Notice that $K^+_{\tau_c} = K^+$, $K^{\#}_{\tau_c} = K^{\#}$ and $K^{\&}_{\tau_c} = K^{\&}$.
Since
$
\ell(K^+) \cap E^* = \ell(K_\tau^+)
$
we have
$K_\tau^+\setminus \ell(K_\tau^+) = (K^+\setminus \ell(K^+)) \cap E^*.$

\begin{lemma} \label{lem:K+AndicorK}
Consider a real linear space $E$ and a nontrivial cone $K \subseteq E$. Then, 
\begin{itemize}
	\item[$1^\circ$] If $K$ is relatively solid, then 
    \begin{align*}
    K^+ \setminus \ell(K^+) & = \{y' \in K^+ \setminus \{0\} \mid \exists\, k \in {\rm icor}\, K:\; y'(k) > 0\}\\
    & \subseteq \{y' \in E' \mid \forall\, k \in {\rm icor}\, K:\; y'(k) > 0\},
    \end{align*}
    and if further $K \subseteq {\rm acl}({\rm icor}\, K)$, then the last inclusion is an equality.
    \item[$2^\circ$] If $K$ is solid, then $K^+$ is pointed and
    \begin{align*}
      K^+ \setminus \{0\}   & = \{y' \in K^+ \setminus \{0\} \mid \exists\, k \in {\rm cor}\, K:\; y'(k) > 0\}\\
      & \subseteq \{y' \in E' \mid \forall\, k \in {\rm cor}\, K:\; y'(k) > 0\},
    \end{align*}
    and if  further  $K \subseteq {\rm acl}({\rm cor}\, K)$, then the last inclusion is an equality.
    \item[$3^\circ$] Assume that $(E, \tau)$ is a real topological linear space. If $K$ is relatively solid, then
    \begin{align*}
     K_\tau^+\setminus \ell(K_\tau^+) & = \{y^* \in K_\tau^+ \setminus \{0\} \mid \exists\, k \in {\rm icor}\, K:\; y^*(k) > 0\}\\
     & \subseteq \{y^* \in E^* \mid \forall\, k \in {\rm icor}\, K:\; y^*(k) > 0\},   
    \end{align*}
    and if further $K \subseteq {\rm cl}_\tau({\rm icor}\, K)$, then the last inclusion is an equality.
\end{itemize}
\end{lemma}

\begin{proof}
\begin{itemize}
	\item[$1^\circ$]
	First, let us check that
	\begin{align*}
	   \ell(K^+) & = \{y' \in E' \mid \forall\, k \in K:\; y'(k) = 0\} \\
	   & = \{y' \in K^+ \mid \forall\, k \in {\rm icor}\, K:\; y'(k) = 0\} \\
	   & = \{y' \in K^+ \mid \exists\, k \in {\rm icor}\, K:\; y'(k) = 0\}.
	\end{align*}
	The first equality and further inclusions of type ``$\subseteq$'' are obvious.  Take some $y' \in K^+$ such that $y'(k) = 0$ for some $k \in {\rm icor}\, K$. On the contrary, assume that there is $\bar k \in K$ with $y'(\bar k) \neq 0$. Due to $k \in {\rm icor}\, K$ and $\bar k, -\bar k \in {\rm aff}\, K = {\rm aff}(K - K)$ there is $\varepsilon > 0$ such that $k + [-\varepsilon, \varepsilon] \bar k \subseteq K$. Then, we get $y'(k + \delta \bar k ) = \delta y'(\bar k) < 0$ for an adequate $\delta \in [-\varepsilon, \varepsilon]$, a contradiction to $y' \in K^+$.  	Thus, we conclude 
	\begin{align*}
	  K^+ \setminus \ell(K^+) 
	   & = \{y' \in K^+ \setminus \{0\} \mid \exists\, k \in K:\; y'(k) > 0\} \\
	   & = \{y' \in K^+ \setminus \{0\} \mid \exists\, k \in {\rm icor}\, K:\; y'(k) > 0\} \\
	   & = \{y' \in K^+ \setminus \{0\}  \mid \forall\, k \in {\rm icor}\, K:\; y'(k) > 0\}\\
	   & = \{y' \in E'  \mid \forall\, k \in {\rm icor}\, K:\; y'(k) > 0\} \cap K^+.
	\end{align*}

	Assume that $K \subseteq {\rm acl}({\rm icor}\, K)$. Then, 
	$
	\{y' \in E' \mid \forall k \in {\rm icor}\, K:\; y'(k) > 0\} \subseteq K^+.
	$
	Indeed, for $y' \in E'$ with $y'(k) > 0$ for all $k \in {\rm icor}\, K \neq \emptyset$ one has $y'(k) \geq 0$ for all $k \in {\rm acl}({\rm icor}\, K) \supseteq K$, i.e., $y' \in K^+$.
	We conclude the validity of the assertion $1^\circ$.
	
    \item[$2^\circ$] Let us show that the solidness of $K$ implies the pointedness of $K^+$. Assume that ${\rm cor}\, K \neq \emptyset$. 
    On the contrary, assume that $\ell(K^+) \neq \{0\}$, i.e., there is $x' \neq 0$ such that $x'(x) = 0$ for all $x \in K$. Then, by $x' \neq 0$ there is $y \in E \setminus K$ with $x'(y) \neq 0$. Pick now $k \in {\rm cor}\, K$. Then, there is $\varepsilon > 0$ such that $k + \varepsilon y \in K$. Thus, we arrive at the contradiction
    $0 = x'(k + \varepsilon y) = x'(k) + \varepsilon x'(y) = \varepsilon x'(y) \neq 0.$
    
    Now, it is easy to see that assertion $2^\circ$ is a direct consequence of assertion  $1^\circ$.
    \item[$3^\circ$]
    The first part of the assertions follows by assertion $1^\circ$ and the fact that $K_\tau^+\setminus \ell(K_\tau^+) = (K^+\setminus \ell(K^+)) \cap E^*$. The proof of the second part is similar to the proof in assertion $1^\circ$ (use $K \subseteq {\rm cl}_\tau({\rm icor}\, K)$ in the topological setting).
\end{itemize}
\end{proof}

 \begin{remark}
 Notice,  Lemma \ref{lem:K+AndicorK} ($1^\circ$) generalizes a result by Khazayel {\rm et al.} \cite[Th. 4.8, Rem. 5.8]{Khazayel2021a} since for a relatively solid, convex cone $K \subseteq E$ one has $K \subseteq {\rm acl}\, K = {\rm acl}({\rm icor}\, K)$ (see  \cite[Lem. 2.5]{Khazayel2021a}).
 \end{remark}

Now, we define a general base concept for cones in linear spaces.

\begin{definition} \label{def:baseKnowngeneral}
	A set $B \subseteq K \setminus \{0\}$ is called a {base} for the cone $K \subseteq E$, if
	$B$ is a nonempty set,  and every $x \in K \setminus \{0\}$ has a unique representation of the form 
	$
	x = \lambda b \mbox{ for some } \lambda > 0 \mbox{ and some } b \in B.
	$
\end{definition}

\begin{remark}
If $B$ is base in the sense of Definition \ref{def:baseKnowngeneral} for $K$, then $K = \mathbb{R}_+ \cdot B$. Notice, a convex base $B$ in the sense of Definition \ref{def:baseKnowngeneral} is used by Jahn \cite[Def. 1.10]{Jahn2011} as a base concept for convex cones.
\end{remark}

In our paper, we will work with the following examples of (not necessarily convex) bases for cones.

\begin{example} \label{def:normbase}
	Assume that $||\cdot||: E \to \mathbb{R}$ is a norm, and $K \neq \{0\}$. A set $B \subseteq K$ given by
	$
	B := \{x \in K \mid ||x|| = 1\}
	$
	is a base in the sense of Definition \ref{def:baseKnowngeneral} for $K$, and is called
	{norm-base} for $K$.
\end{example}

\begin{example} \label{def:seminormbase}
	Assume that $\psi: E \to \mathbb{R}$ is a seminorm. Consider the cone $\tilde{K} := \{x \in K \mid \psi(x) = 0\} = K \cap K_\psi,$
	where
	$
	K_\psi := \{x \in E \mid \psi(x) = 0\}.
	$
	Notice that $K_\psi$ is a convex cone.
	A nonempty set $B \subseteq K$ given by
	$
	B := \{x \in K \mid \psi(x) = 1\}
	$
	is a base in the sense of Definition \ref{def:baseKnowngeneral} for the cone $(K \setminus \tilde{K}) \cup \{0\} = (K \setminus K_\psi) \cup \{0\}$. In general, we have 	
	$
	K = (\mathbb{R}_+ \cdot B) \cup \tilde{K}.
	$ 
	In such a situation we call $B$ a {seminorm-base} for $K$.
\end{example}
	
\begin{example} \label{def:normlikebase}	
	Assume that $\psi: E \to \mathbb{R}$ is a seminorm but not necessarily a norm, and  $K \neq \{0\}$. 	If $\psi$ is positive on $K \setminus \{0\}$ (i.e., $\psi$ behaves on $K$ like a norm), then $\tilde{K} = \{0\}$ and $K = \mathbb{R}_+ \cdot   B$, and so
	$
	B = \{x \in K \mid \psi(x) = 1\}
	$
	is a base in the sense of Definition \ref{def:baseKnowngeneral} for $K$. Then, $B$ is called
	{normlike-base} for $K$. 
\end{example}

\begin{lemma} \label{lem:K=convB}
Assume that a cone $K  \subseteq E$ is generated by some nonempty set $B \subseteq E$, i.e., $K = \mathbb{R}_+ \cdot B$. Then, the following assertions hold:
\begin{itemize}
	\item[$1^\circ$] If $K$ is convex, then $K = \mathbb{R}_+ \cdot {\rm conv}(B)$. 
	\item[$2^\circ$] If ${\rm conv}(B)$ is solid, then ${\rm cor}\, K \subseteq \mathbb{P} \cdot {\rm cor}({\rm conv}(B))$.
	\item[$3^\circ$] If $K$ is convex, then ${\rm cor}\, K \supseteq \mathbb{P} \cdot {\rm cor}({\rm conv}(B))$.
	\item[$4^\circ$] $K = \mathbb{R}_+ \cdot (K \cap B)$.
	\item[$5^\circ$] If $K \neq E$, then ${\rm cor}\, K = \mathbb{P} \cdot (({\rm cor}\, K ) \cap B)$.
	\item[$6^\circ$] $K \setminus \ell(K) = \mathbb{P} \cdot ((K \setminus \ell(K)) \cap B) = \mathbb{P} \cdot ((E \setminus (- K)) \cap B)$.
	\item[$7^\circ$] $K \setminus \{0\} = \mathbb{P} \cdot (B \setminus\{0\}) = \mathbb{P} \cdot ((K \setminus\{0\}) \cap B)$, and if $0 \notin B$, then $K \setminus \{0\} = \mathbb{P} \cdot B$.
\end{itemize}	
\end{lemma}

\begin{proof}
\begin{itemize}
	\item[$1^\circ$] This equality follows easily by the fact that $K$ is a convex cone.
	\item[$2^\circ$] Let ${\rm conv}(B)$ be solid. Take some $\bar x \in {\rm cor}\, K$. Assume by the contrary that $\left(\mathbb{P} \cdot \{\bar x\}\right) \cap {\rm cor}({\rm conv}(B)) = \emptyset$.
	Clearly, then $\left(\mathbb{P} \cdot \{\bar x\}\right) \cap \Omega^1 = \emptyset$ for $\Omega^1 := \mathbb{R}_+ \cdot ({\rm cor}({\rm conv}(B)))$. Thus, for the solid, convex cone $\Omega^1$ and the relatively solid, convex cone $\Omega^2 := \mathbb{R}_+ \cdot \{\bar x\}$ (with ${\rm icor}\, \Omega^2 = \mathbb{P} \cdot \{\bar x\}$) we have $\left({\rm icor}\, \Omega^2\right) \cap ({\rm cor}\, \Omega^1) = \emptyset$.  By the separation theorem in \cite[Cor. 2.28]{Khazayel2021a} there is $x' \in E' \setminus \{0\}$ such that
	$
	x'(\omega^2) \leq 0 \leq x'(\omega^1)$ for all $\omega^2 \in {\rm icor}\, \Omega^2$ and $\omega^1 \in {\rm cor}\,\Omega^1$.
    This also implies
	\begin{equation}
	\label{eq:xbarleq0}   
	x'(\bar x) \leq 0 \leq x'(\omega^1)\quad \mbox{for all } \omega^1 \in {\rm acl}({\rm cor}\,\Omega^1).
    \end{equation}
	  Since ${\rm acl}({\rm cor}\,\Omega^1) =  {\rm acl}\, \Omega^1 \supseteq {\rm acl}({\rm cor}\,({\rm conv}(B))) = {\rm acl}({\rm conv}(B))  \supseteq B$,
	it follows from \eqref{eq:xbarleq0} that $x'(b) \geq 0$ for all $b \in B$. Using the assumption that $K = \mathbb{R}_+ \cdot B$ we infer  $x'(k) \geq 0$ for all $k \in K$, i.e.,
	$x' \in K^+ \setminus \{0\}$. By Lemma \ref{lem:K+AndicorK} ($2^\circ$) it follows that $x'(x) > 0$ for all $x \in {\rm cor}\, K$. In particular, $x'(\bar x) > 0$ for the given $\bar x \in {\rm cor}\, K$, a contradiction to \eqref{eq:xbarleq0}.
	\item[$3^\circ$] First, notice that ${\rm conv}(B) \subseteq K$, hence ${\rm cor}({\rm conv}(B)) \subseteq {\rm cor}\, K$. Due to $\mathbb{P} \cdot {\rm cor}\, K = {\rm cor}\, K$ we get the desired inclusion.
	\item[$4^\circ$] Follows by the fact that $K \cap B = B$.
	\item[$5^\circ$] The inclusion ``$\supseteq$'' is obvious taking into account that $\mathbb{P} \cdot {\rm cor}\, K = {\rm cor}\, K$.
	Let us show the reverse inclusion ``$\subseteq$''. Take some $x \in {\rm cor}\, K$. Then, due to $K = \mathbb{R}_+ \cdot B$ there is
	$b \in B$ and $t \geq 0$ such that $x = t b$. Because $K \neq E$ (i.e., $0 \notin {\rm cor}\, K$) we have $t > 0$. Clearly, $\frac{1}{t} x \in B$ and due to $\mathbb{P} \cdot {\rm cor}\, K = {\rm cor}\, K$ we also get $\frac{1}{t} x \in {\rm cor}\, K$. Thus, we get 
	$
	x \in t \cdot (({\rm cor}\, K ) \cap B) \subseteq \mathbb{P} \cdot (({\rm cor}\, K ) \cap B).
	$
	\item[$6^\circ$] Since $B \subseteq K$ we get $E \setminus (- K) \cap B = (K \setminus \ell(K))  \cap B \subseteq K \setminus \ell(K)$. Further it is known that $\mathbb{P} \cdot (K \setminus \ell(K)) = K \setminus \ell(K)$. Thus, the proof is similar to $5^\circ$.
	\item[$7^\circ$] The proof is similar to $5^\circ$ taking into account that $\mathbb{P} \cdot (K \setminus \{0\}) = K \setminus \{0\}$.
\end{itemize}	
\end{proof}

The next lemma provides some more details on the construction of a (semi)norm and gives some ideas for the construction of bases for nontrivial cones.

\begin{lemma} \label{lem:semnorm_max_properties}
	Consider a function $\psi: E \to \mathbb{R}$, a function $\psi_{\max}: E \to \mathbb{R}$ defined by
	$
	\psi_{\max}(y) := \max\{\psi(y), \psi(-y)\}
	$
	for all $y \in E$, and assume that $K, A \subseteq E$ are nontrivial cones.
	Then, the following assertions hold:
	\begin{itemize}
		\item[$1^\circ$] If $\psi$ is positive homogeneous and subadditive (i.e., $\psi$ is sublinear), then $\psi_{\max}$ is a seminorm.
		\item[$2^\circ$] Assume that $\psi$ satisfies $-K = \{x \in E \mid \psi(x) \leq 0\}$. Then,  $\ell(K) = \{y \in E \mid \psi_{\max}(y) \leq 0\}.$
		Moreover, we have:
		
		$\psi_{\max}$ is nonnegative on $E$ $\iff$ $\ell(K) = \{y \in E \mid \psi_{\max}(y) = 0\}$.
		
		$\psi_{\max}$ is positive on $E \setminus \{0\}$ $\iff$  $K$ is pointed. 
		
		\item[$3^\circ$] Assume that $\psi$ satisfies $-{\rm icor}\, K = \{x \in E \mid \psi(x) < 0\}$, and  $0 \notin {\rm icor}\, K$ (e.g., if $K$ is convex and $K \neq \ell(K)$). Define $P_0 := ({\rm icor}\, K) \cup \{0\}$. Then, $\ell(P_0) \setminus \{0\} = \{y \in E \mid \psi_{\max}(y) < 0\}.$ 
		Furthermore, we have:
		
		$\psi_{\max}$ is nonnegative on $E$ $\iff$  $P_0$ is pointed.  
		
		\item[$4^\circ$] Assume that $\psi$ is sublinear and satisfies $-K = \{x \in E \mid \psi(x) \leq 0\}$. Then, $\psi_{\max}$ is a seminorm. Moreover, we have:
		
		$\ell(K) \cap A = \{0\} \iff$ $B_A = \{y \in A \mid \psi_{\max}(y) = 1\}$ is a normlike-base for $A$.\
		
		$K \neq \ell(K)\iff$ $B_K = \{y \in K \mid \psi_{\max}(y) = 1\}$ is a seminorm-base for $K$.
		
		Furthermore, the following assertions are equivalent:
		(i) $K$ is pointed; (ii) $\psi_{\max}$ is a norm; (iii) $B_K$ is a norm-base for $K$; (iv) $B_K$ is a normlike-base for $K$.
	\end{itemize}
\end{lemma}

\begin{proof}
	\begin{itemize}
		\item[$1^\circ$] This assertion is well-known, see \cite[Lem 2.1.29]{TammerWeidner2020}.
		\item[$2^\circ$] If $-K = \{x \in E \mid \psi(x) \leq 0\}$, then $K = \{x \in E \mid \psi(-x) \leq 0\}$. Thus, we get
		\begin{align*}
		\{y \in E \mid \psi_{\max}(y) \leq 0\}
		& = \{y \in E \mid \psi(y) \leq 0, \psi(-y) \leq 0\}\\
		& = \{y \in E \mid \psi(y) \leq 0\} \cap \{y \in E \mid \psi(-y) \leq 0\}\\
		& = (-K)\cap K
		 = \ell(K).
		\end{align*}
		Clearly, $\psi_{\max}$ is nonnegative on $E$ $\iff$ $\{y \in E \mid \psi_{\max}(y) = 0\} = \ell(K)$, while 
		$K$ is pointed $\iff$ $\{y \in E \mid \psi_{\max}(y) > 0\} = E \setminus \{0\}$. 
		
		\item[$3^\circ$] 
		Since $-{\rm icor}\, K = \{x \in E \mid \psi(x) < 0\}$ and $0 \notin {\rm icor}\, K$ we derive
		\begin{align*}
		\{y \in E \mid \psi_{\max}(y) < 0\}
		& = \{y \in E \mid \psi(y) < 0, \psi(-y) < 0\}\\
		& = \{y \in E \mid \psi(y) < 0\} \cap \{y \in E \mid \psi(-y) < 0\}\\
		& = (-{\rm icor}\,K)\cap ({\rm icor}\,K) = \ell(P_0) \setminus \{0\}.
		\end{align*}
		Thus,  $P_0$ is pointed $ \iff \{y \in E \mid \psi_{\max}(y) < 0\} = \emptyset$ (i.e., $\psi_{\max}$ is nonnegative).

		\item[$4^\circ$] By $1^\circ$, $\psi_{\max}$ is a seminorm. 
		Since
		$$
		\{y \in A \mid \psi_{\max}(y) = 0\} = \{y \in E \mid \psi_{\max}(y) = 0\} \cap A \overset{2^\circ}{=} \ell(K) \cap A,
		$$
		we have $\ell(K) \cap A = \{0\} \iff$ $B_A$ is a normlike-base for $A$. 
		Moreover,
		$$
		K \setminus \ell(K) \overset{2^\circ}{=} K \setminus \{y \in E \mid \psi_{\max}(y) = 0\} = \{y \in K \mid \psi_{\max}(y) > 0\},
		$$ 
		hence $K \setminus \ell(K) \neq \emptyset \iff B_K = \{y \in K \mid \psi_{\max}(y) = 1\} \neq \emptyset$ (by the cone property of $K$, and by the positive homogeneity of $\psi$).
		Thus, $K \neq \ell(K)\iff$ $B_K$ is a seminorm-base for $K$. 
		
		Let us prove the remaining equivalences. (i) $\Longrightarrow$ (ii) follows by $1^\circ$ and $2^\circ$. (ii) $\Longrightarrow$ (iii) $\Longrightarrow$ (iv) follow by the definitions of (norm-; normlike-)bases. It remains to show that (iv) $\Longrightarrow$ (i). If $K$ is not pointed, then $\{y \in E \mid \psi_{\max}(y) = 0\} \overset{2^\circ}{=} \ell(K) \neq \{0\}$, and so $B_K$ is not a normlike-base for $K$. 
	\end{itemize}
\end{proof}

\begin{example}Let $K$ be a solid, convex cone in $E$.
Well known separating functions $\psi$ (respectively, scalarization functions in the context of vector optimization) are (under certain assumptions) sublinear and satisfy cone representation properties $-K = \{x \in E \mid \psi(x) \leq 0\}$ and $-{\rm cor}\, K = \{x \in E \mid \psi(x) < 0\}$,
for instance the functions by Gerstewitz \cite{Gerstewitz1983},  Hiriart-Urruty \cite{HiriartUrruty}, and Bishop and Phelps \cite{BP1962} / Kasimbeyli \cite{Kasimbeyli2010} (see also Bouza, Quintana and Tammer \cite{BouzaQuintanaTammer},  Göpfert {\rm et al.} \cite[Sec. 2.3]{GoeRiaTamZal2003}, Ha \cite[Sec. 3.1 and 3.2]{Ha2022}, Jahn \cite{Jahn2023}, Tammer and Weidner \cite{TammerWeidner2020}, Zaffaroni \cite[Prop. 3.2]{Zaffaroni2003}).
\end{example}

In the following, we like to present a first theorem where we derive a representation for the algebraic interior of the topological dual cone $K^{+}_\tau$.

\begin{theorem} \label{th:cor_dual_cone_compact}
	Assume that $(E, \mathcal{F}, \tau)$ is a real separated locally convex space, $K \subseteq E$ is a nontrivial cone, $\psi$ be a seminorm, and  $B_{K}$ is a $\tau$-compact normlike-base of $K$. Then,
   	$
   	{\rm cor}\, K^{+}_\tau  = \{x^* \in E^* \mid \forall\, x \in B_K:\;  x^*(x) > 0\} = K^{\#}_\tau.
   	$
\end{theorem}

\begin{proof} First we prove that ${\rm cor}\, K^{+}_\tau \supseteq K_\tau^{\#}$. 
	Take some $x^* \in K_\tau^{\#}$ (i.e., $x^*(x) > 0$ for all $x \in B_K$) and an arbitrary $y^* \in E^*$.
	We are going to show that 
	\begin{equation}
	    \label{eq:corK+CorProperty1}
	    x^* + [0, \varepsilon] \cdot y^* \subseteq K^{+}_\tau \quad \mbox{for some }\varepsilon > 0.
	\end{equation}
	Since $B_K$ is $\tau$-compact and $x^*$ and $y^*$ are $\tau$-continuous, there are $\delta, C > 0$ such that $x^*(x) \geq \delta > 0$ and $|y^*(x)| \leq C$ for all $x \in B_K$ (by a general version of the well-known Weierstraß theorem, see Jahn \cite[Th. 3.26]{Jahn2011}). Define $\varepsilon := \frac{\delta}{C} \, (> 0)$. Then, for any $x \in B_K$ and $\bar \varepsilon \in [0, \varepsilon]$, we get
	$
	(x^* + \bar \varepsilon  y^*)(x) = x^*(x) + \bar \varepsilon  y^*(x) \geq \delta + \bar \varepsilon  y^*(x)  \geq \delta - \bar \varepsilon  C \geq 0. 
	$
	This shows that \eqref{eq:corK+CorProperty1} is valid, i.e., $x^* \in {\rm cor}\, K^{+}_\tau$.
	
	Now, we are going to prove ${\rm cor}\, K^{+}_\tau \subseteq K_\tau^{\#}$. Basically we adapt the ideas by Jahn \cite[Lem. 1.25]{Jahn2011} that he used for proving the algebraic counterpart ${\rm cor}\, K^+ \subseteq K^\#$. 
	Take some $x^* \in {\rm cor}\, K^{+}_\tau$. On the contrary suppose that $x^* \notin K_\tau^{\#}$, hence there is $k \in \mathbb{P} \cdot B_K = K \setminus \{0\}$ with $x^*(k) \leq 0$. Since the topological dual space $E^*$ separates elements in $E$, there is $y^* \in E^*$ with $y^*(k) < y^*(0) = 0$. Then, we conclude that
	$
	(x^* + \varepsilon y^*)(k) = x^*(k) + \varepsilon y^*(k) < 0$ for all $ \varepsilon > 0$,
	a contradiction to $x^* \in{\rm cor}\, K^+_\tau$.	
\end{proof}

\begin{remark} \label{rem:corK+=intK+BanachSpace}
Assume that $(E, \mathcal{F}, \tau)$ is a real separated locally convex space, and $K \subseteq E$ is a nontrivial cone.
Since the dual cone $K^{+}_\tau$ is a convex set in $E^*$, we know that ${\rm int}\, K^{+}_\tau \neq \emptyset$ implies ${\rm int}\, K^{+}_\tau = {\rm cor}\, K^{+}_\tau$. 
Here, ``${\rm int}$'' denotes the interior w.r.t the weak$^*$ topology in $E^*$ (or the norm topology in $E^*$ in a setting of normed spaces).
Moreover, since $K^{+}_\tau$ is closed and convex, we have ${\rm int}\, K^{+}_\tau = {\rm cor}\, K^{+}_\tau$ in a real Banach space $E^*$ (e.g., if $E$ is a real normed space).
\end{remark}
 
\begin{remark} \label{rem:Krein-Rutman}
    Theorem \ref{th:cor_dual_cone_compact} provides sufficient conditions for the validity of the equality ${\rm cor}\, K^{+}_\tau = K^\#_\tau$. Notice that the nontrivial cone $K$ needs not to be convex.
    In the case that $K$ is a convex cone, there are further results known in the literature concerning the nonemptyness of $K_\tau^\#$ (respectively, $K_\tau^\&$) and relationships between $K_\tau^\#$ and generalized interiors of $K_\tau^+$ (see, e.g., Bot, Grad and Wanka \cite[Prop. 2.1.1]{BotGradWanka},  Jahn \cite[Th. 3.38, Lem. 3.21]{Jahn2011}, Khan, Tammer and Z\u{a}linescu \cite[Th. 2.4.7]{Khanetal2015},  Khazayel {\rm et al.} \cite[Th. 4.1]{Khazayel2021a}). 
\end{remark}

\section{Nonlinear Separating Functions and Augmented Dual Cones}

Assume that $(E, \tau)$ is a real topological linear space,  $K \subseteq E$ is a nontrivial cone, and $\psi: E \to \mathbb{R}$ is a seminorm. For any $(x', \alpha) \in E' \times \mathbb{R}_+$, we consider the separating function $\varphi_{x', \alpha} : E \to \mathbb{R}$ introduced in \eqref{f16} by
\begin{equation}\label{f31}
\varphi_{x', \alpha}(y)   = x'(y) + \alpha \psi(y) \quad \mbox{for all } y \in E.
\end{equation}

\begin{remark} \label{rem_properties_sep_fcn}
The separating function $\varphi_{x', \alpha}$ in \eqref{f31} (which is 
a combination of a seminorm $\psi$ with a linear function $x' \in E'$) can be seen as a generalization of the function \eqref{f12} associated to so-called Bishop-Phelps cones (which is 
a combination of a norm $||\cdot||$ with a linear continuous function $x^* \in E^*$). For more details, we refer the reader to the works by Bishop and Phelps \cite{BP1962}, Jahn \cite[Ex. 2.1]{Jahn2022} \cite[Rem. 2.1]{Jahn2023}, Kasimbeyli \cite{Kasimbeyli2010}, and to Section \ref{sec:bp_cones} of our paper. 
\end{remark}

Following the definitions of so-called augmented dual cones by Kasimbeyli \cite{Kasimbeyli2010}, we define an (algebraic) augmented dual cone by
$$
K^{a+} := \{(x', \alpha) \in K^+ \times \mathbb{R}_+ \mid \forall\, y \in K:\; x'(y) - \alpha \psi(y) \geq 0\}
$$
as well the following further sets
\begin{align*}
K^{a\circ} & := \{(x', \alpha) \in (K^+ \setminus \ell(K^+)) \times \mathbb{R}_+ \mid \forall\, y \in {\rm icor}\, K:\;  x'(y) - \alpha \psi(y)  > 0\},\\
K^{a\#} & := \{(x', \alpha) \in K^\# \times \mathbb{R}_+ \mid \forall\, y \in K \setminus \{0\}:\;  x'(y) - \alpha \psi(y) > 0\},\\
K^{a\&} & := \{(x', \alpha) \in K^\& \times \mathbb{R}_+ \mid \forall\, y \in K \setminus \ell(K):\;  x'(y) - \alpha \psi(y) > 0\}.
\end{align*}

\begin{remark}  \label{rem:adCones_Bases}
In the definitions of $K^{a+}$, $K^{a\#}$ and $K^{a\&}$ one can also write $E'$ (or $K^+$) instead of $K^+$, $K^{\#}$ and $K^\&$, respectively. This is mainly due to the fact that for $(x', \alpha) \in E' \times \mathbb{R}_+$ we have $x'(y) \geq x'(y) - \alpha \psi(y)$ for all $y \in E$.
Furthermore, in the definition of $K^{a\circ}$ one can replace $K^+ \setminus \ell(K^+)$ also by $K^+$, since
$$
\{(x', \alpha) \in \ell(K^+) \times \mathbb{R}_+ \mid \forall\, y \in {\rm icor}\, K:\;  x'(y) - \alpha \psi(y)  > 0\} = \emptyset,
$$
and further by $E'$ if $K \subseteq {\rm acl}({\rm icor}\, K)$ (e.g., if $K$ is a relatively solid, convex cone).

It is a simple observation (with view on Lemma \ref{lem:K=convB}) that for a normlike-base $B_K$ of $K$ we have
\begin{align*}
K^{a+} & = \{(x', \alpha) \in K^+ \times \mathbb{R}_+ \mid \forall\, y \in B_K:\; x'(y) \geq  \alpha\},\\
K^{a\#} & = \{(x', \alpha) \in K^\# \times \mathbb{R}_+ \mid \forall\, y \in B_K:\;  x'(y) > \alpha\},\\
K^{a\&} & = \{(x', \alpha) \in K^\& \times \mathbb{R}_+ \mid \forall\, y \in B_K \cap K \setminus \ell(K):\;  x'(y) > \alpha\},
\end{align*}
and if $K$ is solid (hence $K^+$ is pointed), then
\begin{align*}
K^{a\circ} & = \{(x', \alpha) \in (K^+ \setminus \{0\}) \times \mathbb{R}_+ \mid \forall\, y \in  B_K \cap {\rm cor}\,K:\;  x'(y) > \alpha\}.
\end{align*}
\end{remark}

It is convenient, to introduce topological counterparts of the algebraic augmented dual cone $K^{a+}$ and the corresponding sets $K^{a\circ}$, $K^{a\#}$ and $K^{a\&}$ in the following way:
\begin{align*}
	K^{a+}_\tau & := K^{a+} \cap (E^* \times \mathbb{R}), & K^{a\#}_\tau & := K^{a\#} \cap (E^* \times \mathbb{R}),\\
    K^{a\circ}_\tau &:= K^{a\circ} \cap (E^* \times \mathbb{R}), &
	K^{a\&}_\tau & := K^{a\&} \cap (E^* \times \mathbb{R}).
\end{align*}

\begin{remark} Recall that the following properties are valid:
	\begin{itemize}
		\item[$1^\circ$] ${\rm int}_\tau\, K \subseteq {\rm cor}\,K  \subseteq {\rm icor}\,K$, and if $K$ is convex with ${\rm int}_\tau\, K \neq \emptyset$, then ${\rm icor}\,K = {\rm cor}\,K = {\rm int}_\tau\, K$.
		\item[$2^\circ$] If $E$ is locally convex, $K$ is convex, and ${\rm rint}_\tau\, K \neq \emptyset$, then ${\rm icor}\,K = {\rm rint}_\tau\, K$.
		\item[$3^\circ$] $K^{+} \cap E^* = K^{+}_\tau$, $(K^+ \setminus \ell(K^+)) \cap E^* = K^+_\tau \setminus \ell(K^+_\tau)$, $K^{\#} \cap E^* = K^{\#}_\tau$ and $K^{\&} \cap E^* = K^{\&}_\tau$.
	\end{itemize}
\end{remark}

\begin{remark} \label{rem_properties_sep_fcn_2}
We like to collect some useful properties of the separating function $\varphi_{x', \alpha}$ with  $(x', \alpha) \in E' \times \mathbb{R}_+$:
\begin{itemize}
    \item $\varphi_{x', \alpha}$ is well-defined and sublinear (hence convex).
    \item If $x' \in E^*$ and $\psi$ is $\tau$-continuous, then $\varphi_{x', \alpha}$ is $\tau$-continuous.
    \item The zero lower-level set of $\varphi_{x', \alpha}$, namely the set 
    $C_{x', \alpha}^{\leq} := \{x \in E \mid \varphi_{x', \alpha}(x) \leq 0\},$
    is a convex cone, and if $\varphi_{x', \alpha}$ is $\tau$-continuous, then $C_{x', \alpha}$ is $\tau$-closed.
    \item If the strict zero lower-level set of $\varphi_{x', \alpha}$, namely the set
    $C_{x', \alpha}^{<} := \{x \in E \mid \varphi_{x', \alpha}(x) < 0\},$
    is nonempty (e.g., if $(x', \alpha) \in K^{a\#}$), then 
    $
    {\rm cor}\, C_{x', \alpha}^{\leq} = C_{x', \alpha}^{<}.
    $
    \item If $C_{x', \alpha}^{<}$ is nonempty and $\varphi_{x', \alpha}$ is $\tau$-continuous, then 
    $
    {\rm int}_\tau\, C_{x', \alpha}^{\leq} = C_{x', \alpha}^{<}.
    $
\end{itemize}

As mentioned in the Introduction, we will focus in our upcoming Section \ref{sec:cone_separation_lcs} on the separation of two (not necessarily convex) cones by a (convex) cone $C \subseteq E$. More precisely, we like to find $(x', \alpha) \in E' \times \mathbb{R}_+$ such that $C = C_{x', \alpha}^{\leq}$ is a separating cone. 

\end{remark}

In the next lemma, we study relationships between the augmented dual cone $K_\tau^{a+}$ and the sets  $K_\tau^{a\#}$, $K_\tau^{a\&}$ and $K_\tau^{a\circ}$. 

\begin{lemma} 	\label{lem:Kaw_prop}
	Assume that $E$ is a real linear space, and $K$ is a nontrivial cone with normlike-base $B_K$. Then, the following assertions hold:
	\begin{itemize}
		\item[$1^\circ$] $K^{a\#} \subseteq K^{a\&} \cup K^{a+}$ and $K_\tau^{a\#} \subseteq K_\tau^{a\&} \cup K_\tau^{a+}$.
		\item[$2^\circ$] If ${\rm icor}\,K \subseteq K \setminus \ell(K)$ (e.g., if $K$ is convex with $K \neq \ell(K)$), then $K^{a\&} \subseteq K^{a\circ}$ and $K_\tau^{a\&} \subseteq K_\tau^{a\circ}$.
		\item[$3^\circ$] If $0 \notin {\rm icor}\,K$ (e.g., if $K$ is convex with $K \neq \ell(K)$), then $K^{a\#} \subseteq K^{a\circ}$ and $K_\tau^{a\#} \subseteq K_\tau^{a\circ}$.
		\item[$4^\circ$] Let $(E, \tau)$ be a real topological linear space, and $\psi$ be $\tau$-continuous. If $K \subseteq {\rm cl}_\tau({\rm icor}\,K)$ (e.g., if $K$ is $\tau$-solid and convex), then $K_\tau^{a\circ} \subseteq K_\tau^{a+}$. 
		\item[$5^\circ$] Let $(E, \tau)$ be a real topological linear space, and $\psi$ be $\tau$-continuous. If $K \subseteq {\rm cl}_\tau(K \setminus \ell(K))$ (e.g., if $K$ is $\tau$-solid and convex with $K \neq \ell(K)$), then $K_\tau^{a\&} \subseteq K_\tau^{a+}$.  
	\end{itemize}	
\end{lemma}

\begin{proof} First, notice that in assertions $1^\circ$, $2^\circ$ and $3^\circ$ the topological statement is derived from the algebraic statement by an intersection with $E^* \times \mathbb{R}$.
	\begin{itemize}
	\item[$1^\circ$] Since $K \setminus \ell(K) \subseteq K \setminus \{0\}$ we get $K^{a\#} \subseteq K^{a\&}$ (taking into account the first part of Remark \ref{rem:adCones_Bases}). The inclusion $K^{a\#} \subseteq K^{a+}$ is obvious by noting that for $(x', \alpha) \in K^+ \times \mathbb{R}_+$ we have $x'(0) + \alpha \psi(0) = 0$. 
	\item[$2^\circ$] Obviously, $K^{a\&} \subseteq K^{a\circ}$ if ${\rm icor}\,K \subseteq K \setminus \ell(K)$.
	Notice, if $K$ is convex, then $K \neq \ell(K) \iff 0 \notin {\rm icor}\, K \iff {\rm icor}\, K \subseteq K \setminus \ell(K)$ (see Khazayel {\rm et al.} \cite[Lem. 2.9]{Khazayel2021a}).
	\item[$3^\circ$] If $0 \notin {\rm icor}\,K$, then ${\rm icor}\, K \subseteq K \setminus \{0\}$, hence $K^{a\#} \subseteq K^{a\circ}$.
	\item[$4^\circ$] Fix some $(x', \alpha) \in K^+ \times \mathbb{R}_+$.
	From $x'(y) - \alpha \psi(y)  > 0$ for all $y \in {\rm icor}\, K$ we get $x'(y) - \alpha \psi(y)  \geq 0$ for all $y \in {\rm cl}_\tau({\rm icor}\, K) \supseteq K$, i.e., $K_\tau^{a\circ} \subseteq K_\tau^{a+}$. 
	
	If $K$ is $\tau$-solid and convex, then ${\rm int}_\tau\, K = {\rm cor}\, K = {\rm icor}\, K$, hence 
	$K \subseteq {\rm cl}_\tau\, K = {\rm cl}_\tau({\rm int}_\tau\, K) =  {\rm cl}_\tau({\rm icor}\, K).$
	
	\item[$5^\circ$]
	
	Fix some $(x', \alpha) \in K^+ \times \mathbb{R}_+$.
	From $x'(y) - \alpha \psi(y)  > 0$ for all $y \in  K \setminus \ell(K)$ we get $x'(y) - \alpha \psi(y)  \geq 0$ for all $y \in {\rm cl}_\tau(K \setminus \ell(K)) \supseteq K$, i.e., $K_\tau^{a\&} \subseteq K_\tau^{a+}$.
	
	If $K$ is $\tau$-solid and convex with $K \neq \ell(K)$, then ${\rm int}_\tau\, K = {\rm cor}\, K = {\rm icor}\, K$, hence 
	$K \subseteq {\rm cl}_\tau\, K = {\rm cl}_\tau({\rm int}_\tau\, K) = {\rm cl}_\tau({\rm icor}\, K) \subseteq {\rm cl}_\tau(K \setminus \ell(K)).$
	
	\end{itemize}
\end{proof}

Next, we state some main properties of the algebraic augmented dual cone $K^{a+}$.

\begin{lemma} \label{lem:properties_augmented_dual_cones}
	Assume that $E$ is a real linear space, and $K$ is a nontrivial cone.
	Then, the following assertions hold:
    \begin{itemize}
    	\item[$1^\circ$] $K^+ \times \{0\} \subseteq K^{a+}$, $K^\#  \times \{0\} \subseteq K^{a\#}$, $K^\&  \times \{0\} \subseteq K^{a\&}$,  $(K^+ \setminus \ell(K^+)) \times \{0\} \subseteq K^{a\circ}$.
    	\item[$2^\circ$] {\bf Cone property:} $\mathbb{P} \cdot K^{a\theta} \subseteq K^{a\theta}$ for all $\theta \in \{+, \circ, \#, \&\}$,
    	and $0 \in K^{a+}$ (i.e., $K^{a+}$ is a cone). 
    	If ${\rm icor}\, K \neq \emptyset$ ($K \neq \{0\}$, respectively $K \neq \ell(K)$), then $0 \notin K^{a\circ}$ ($0 \notin K^{a\#}$, respectively $0 \notin K^{a\&}$).
    	
    	\item[$3^\circ$] {\bf Convexity:} $K^{a\theta}$ is convex (i.e., $K^{a\theta} + K^{a\theta} \subseteq K^{a\theta}$) for all $\theta \in \{+, \circ, \#, \&\}$.
    	
    	\item[$4^\circ$] {\bf  Pointedness:} $\ell(K^{a+}) = \ell(K^+) \times \{0\}$, hence $K^{a+}$ is pointed $\iff$ $K^+$ is pointed.
    	Moreover, if there is $y \in K$ with $\psi(y) > 0$ (e.g., if $K$ has a normlike-base), then 
    	$K^{a+} \setminus \ell(K^{a+}) = K^{a+} \cap (K^+ \setminus \ell(K^+) \times \mathbb{P})$.
    
    	\item[$5^\circ$] {\bf  Nontriviality:} If there is $y \in K$ with $\psi(y) > 0$, then $K^{a+}$ is nontrivial $\iff$ $K^+ \neq \{0\}$.
\end{itemize}
\end{lemma}

\begin{proof}

\begin{itemize}
	\item[$1^\circ$] The first three inclusions are obvious by taking a look on the definitions of $K^+$, $K^\#$ and $K^\&$. The last inclusion is a consequence of Lemma \ref{lem:K+AndicorK} ($1^\circ$).	
	\item[$2^\circ$] This assertion follows easily from the definitions of $K^{a\theta}$, $\theta \in \{+, \circ, \#, \&\}$.
	\item[$3^\circ$] Due to assertion $2^\circ$, in order to show that $K^{a\theta}$ is convex it is enough to prove the property $K^{a\theta} + K^{a\theta} \subseteq K^{a\theta}$ for all $\theta \in \{+, \circ, \#, \&\}$. The proof is straightforward taking into account the first part of Remark \ref{rem:adCones_Bases}.
    
	\item[$4^\circ$] It is easy to observe that $\ell(K^{a+}) = \ell(K^+) \times \{0\}$.
	As a direct consequence we get that $K^{a+}$ is pointed if and only if $K^+$ is pointed. 
	Since there is $y \in K$ with $\psi(y) > 0$, we derive 
	\begin{align*}
	& \{(x', \alpha) \in \ell(K^+) \times \mathbb{P}  \mid \forall\, y \in K:\; x'(y) - \alpha \psi(y) \geq 0\}\\
	& = \{(x', \alpha) \in \ell(K^+) \times \mathbb{P}  \mid \forall\, y \in K:\; \alpha \psi(y) \leq 0\}  = \emptyset,
	\end{align*}
	and so
	$$K^{a+} \setminus \ell(K^{a+}) = \{(x', \alpha) \in K^+ \setminus \ell(K^+) \times \mathbb{P} \mid \forall\, y \in K: x'(y) - \alpha \psi(y) \geq 0 \}.$$
	\item[$5^\circ$] Obviously, $K^{a+} \neq E' \times \mathbb{R}_+$. Then, the proof of $K^{a+} \neq \{(0,0)\} \iff K^+ \neq \{0\}$ (if there is $y \in K$ with $\psi(y) > 0$) is straightforward.
\end{itemize}
\end{proof}

The following topological counterpart (with properties of $K_\tau^{a+}$) to the previous lemma (with properties of $K^{a+}$) holds true. 

\begin{lemma} \label{lem:top_properties_augmented_dual_cones}
	Assume that $(E, \tau)$ is a real topological linear space, and $K$ is a nontrivial cone.
	Then, the following assertions hold:
	\begin{itemize}
	\item[$1^\circ$] $K_\tau^+ \times \{0\} \subseteq K_\tau^{a+}$, $K_\tau^\#  \times \{0\} \subseteq K_\tau^{a\#}$, $K_\tau^\&  \times \{0\} \subseteq K_\tau^{a\&}$, $(K_\tau^+ \setminus \ell(K_\tau^+)) \times \{0\} \subseteq K_\tau^{a\circ}$.
	\item[$2^\circ$] {\bf Cone property:} $\mathbb{P} \cdot K_\tau^{a\theta} = K_\tau^{a\theta}$ for all $\theta \in \{+, \circ, \#, \&\}$,
	and $0 \in K_\tau^{a+}$ (i.e., $K_\tau^{a+}$ is a cone). 
	If ${\rm icor}\, K \neq \emptyset$ ($K \neq \{0\}$, respectively $K \neq \ell(K)$), then $0 \notin K_\tau^{a\circ}$ ($0 \notin K_\tau^{a\#}$, respectively $0 \notin K_\tau^{a\&}$).
	
	\item[$3^\circ$] {\bf Convexity:} $K_\tau^{a\theta}$ is convex (i.e., $K_\tau^{a\theta} + K_\tau^{a\theta} \subseteq K_\tau^{a\theta}$) for all $\theta \in \{+, \circ, \#, \&\}$.
	
	\item[$4^\circ$] {\bf  Pointedness:} $\ell(K_\tau^{a+}) = \ell(K_\tau^+) \times \{0\}$, hence $K_\tau^{a+}$ is pointed $\iff$ $K_\tau^+$ is pointed.
	Moreover, if there is $y \in K$ with $\psi(y) > 0$ (e.g., if $K$ has a normlike-base), then
	$K_\tau^{a+} \setminus \ell(K_\tau^{a+}) = K_\tau^{a+} \cap (K_\tau^+ \setminus \ell(K_\tau^+) \times \mathbb{P})$.
		
	\item[$5^\circ$] {\bf  Nontriviality:}  If there is $y \in K$ with $\psi(y) > 0$, then $K^{a+}_\tau$ is nontrivial $\iff$ $K^+_\tau \neq \{0\}$.
\end{itemize}
\end{lemma}

\begin{proof}
The proof is analogous to the proof of Lemma \ref{lem:properties_augmented_dual_cones}.
\end{proof}

Concerning the algebraic interior of the topological augmented dual cone $K^{a+}_\tau$ we have the following result (see also Theorem \ref{th:cor_dual_cone_compact}):

\begin{theorem} \label{th:cor_augmented_dual_cone_compact}
Assume that $(E, \tau)$ is a real topological linear space, and $K \subseteq E$ is a nontrivial cone with $\tau$-compact normlike-base $B_{K}$. Then,
\begin{align*}
{\rm cor}\, K^{a+}_\tau & = \{(x^*, \alpha) \in K^+_\tau \times \mathbb{P} \mid \forall\, y \in B_K:\;  x^*(y) > \alpha\} = K^{a\#}_\tau \cap (E^*\times \mathbb{P}).
\end{align*}
\end{theorem}

\begin{proof}

 The second equality is obvious. We are going to prove the first equality.
	
	Take some $(x^*, \alpha) \in {\rm cor}\, K^{a+}_\tau \subseteq K^+_\tau \times \mathbb{R}_+$. Clearly, we have $\alpha > 0$, otherwise there is $\varepsilon > 0$ such that $(x^*, -\varepsilon) = (x^*, \alpha) + \varepsilon (0, -1) \in K^{a+}_\tau$, a contradiction. There is $\varepsilon > 0$ such that
	$
	(x^*, \alpha) + \varepsilon (0, \alpha) = (x^*, (1+ \varepsilon)\alpha) \in K^{a+}_\tau.
	$
	Thus, 
	$
	x^*(y) \geq (1+ \varepsilon)\alpha > \alpha \quad \mbox{for all } y \in B_K.
	$ 
	We conclude that $(x^*, \alpha) \in K^{a\#}_\tau \cap (E^*\times \mathbb{P})$.
	
	Conversely, take $(x^*, \alpha) \in K^+_\tau \times \mathbb{P}$ such that $x^*(x) > \alpha$ for all $x \in B_K$. Take an arbitrary $(y^*, \beta) \in E^* \times \mathbb{R}$. We are going to prove that 
    \begin{equation}
      \label{eq:corKa+CorProperty}
    	(x^*, \alpha) + [0, \varepsilon] \cdot (y^*, \beta) \subseteq K_\tau^{a+} \quad \mbox{for some } \varepsilon > 0.
    \end{equation}
	Because $B_K$ is $\tau$-compact and $x^*$ is $\tau$-continuous there is $\delta >0 $ such that for any $x \in B_K$ we have $x^*(x) \geq \delta > \alpha$, respectively, $x^*(x) - \alpha \geq \delta - \alpha > 0$. 
	Since $B_K$ is $\tau$-compact and $y^*$ is $\tau$-continuous there is $C > 0$ such that $|y^*(x) - \beta| \leq C$ for all $x \in B_K$. Hence, for any $\bar \varepsilon \in [0, \frac{\delta - \alpha}{C}]$ and $x \in B_K$, we get 
	$$
	(x^* + \bar \varepsilon y^*)(x) - \alpha - \bar \varepsilon \beta 
	 \geq \delta - \alpha + \bar \varepsilon (y^*(x) - \beta) \geq \delta - \alpha - \bar \varepsilon C \geq 0.
	$$
    Moreover, if $\beta \geq 0$, then $\alpha + \bar \varepsilon \beta \geq 0$ for all $\bar \varepsilon \geq 0$; otherwise, if $\beta < 0$, then
	$\alpha + \bar \varepsilon \beta \geq 0$ for all $\bar \varepsilon \in [0, -\frac{\alpha}{\beta}]$. 
	Now, define $\varepsilon := \min \{\frac{\delta - \alpha}{C}, -\frac{\alpha}{\beta}\}$. 
	Clearly, 
	$(x^* + \bar \varepsilon y^*)(x) \geq \alpha + \bar \varepsilon \beta \geq 0$ for all $x \in B_K$ and all $\bar \varepsilon \in [0, \varepsilon]$, hence $x^* + [0, \varepsilon] \cdot y^* \subseteq K^+_\tau$.  We conclude that \eqref{eq:corKa+CorProperty} is valid, which shows that $(x^*, \alpha) \in {\rm cor}\, K^{a+}_\tau$.
\end{proof}

For a not necessarily convex cone $K$, the following result gives conditions for the existence of elements of the set $K^{a+} \cap (E' \times \mathbb{P})$ and $K^{a+}_\tau \cap (E^* \times \mathbb{P})$, respectively.

\begin{lemma} \label{lem:KaPlushasPosElement}
Assume that $E$ is a real linear space, and $K$ is a nontrivial cone with normlike-base $B_{K}$.  Then, the following assertions hold:
	\begin{itemize} 
    	\item[$1^\circ$]  If $c := {\rm inf}_{x \in B_{K}}\, x'(x) > 0$ for some $x' \in E'$, then $(x', c)\in  K^{a+} \cap (E' \times \mathbb{P})$.
		\item[$2^\circ$]  Let $(E, \tau)$ be a real topological linear space, and $B_{K}$ be  $\tau$-compact. Then, $(x^*, c)$ for $x^* \in K^{\#}_\tau$ and $c := {\rm inf}_{x \in B_{K}}\, x^*(x)$ is belonging to $K^{a+}_\tau \cap (E^* \times \mathbb{P})$.
	\end{itemize}	     
\end{lemma}

\begin{proof}
    \begin{itemize} 
        	\item[$1^\circ$] Take some $x' \in E'$ with $c := {\rm inf}_{x \in B_{K}}\, x'(x) > 0$. Clearly, $x'(x) \geq c > 0$ for all $x \in B_K$, i.e., $(x', c)\in  K^{a+} \cap (E' \times \mathbb{P})$.
        	\item[$2^\circ$] Assume that $x^* \in K_\tau^\#$, i.e., $x^*(y) > 0$ for all $y \in B_K$. If $B_K$ is $\tau$-compact, then $c = {\rm min}_{x \in B_{K}}\, x^*(x) > 0$.
        	By $1^\circ$ we then get $(x^*, c) \in K^{a+}_\tau \cap (E^* \times \mathbb{P})$.
     \end{itemize}   	
\end{proof}

\begin{lemma} \label{lem:KaSharphasPosElements}
Assume that $E$ is a real linear space, and $K$ is a nontrivial cone with normlike-base $B_{K}$.  Then, the following assertions hold:
	\begin{itemize} 
    	\item[$1^\circ$]  Let $(E, \tau)$ be a real topological linear space. If $(x^*, \alpha) \in  K^{a+}_\tau \cap (E^* \times \mathbb{P})$, then $(x^*, \alpha -\varepsilon) \in K^{a\#}_\tau$ for all $\varepsilon \in (0, \alpha]$.
		\item[$2^\circ$] If $(x', \alpha) \in  K^{a+} \cap (E' \times \mathbb{P})$, then $(x', \alpha -\varepsilon) \in K^{a\#}$ for all $\varepsilon \in (0, \alpha]$.
	\end{itemize}	
\end{lemma}

\begin{proof}
	\begin{itemize}
		\item[$1^\circ$] Take some $(x^*, \alpha) \in  K_\tau^{a+}$ with $\alpha > 0$, i.e., $x^*(x) \geq \alpha > 0$ for all $x \in B_K$. Clearly, for any $\varepsilon \in (0, \alpha]$, we have $x^*(x) > \alpha - \varepsilon \geq 0 $ for all $x \in B_K$, and so $(x^*, \alpha -\varepsilon) \in  K_\tau^{a\#}$.
		\item[$2^\circ$] Follows by $1^\circ$ (applied to the convex core topology $\tau := \tau_c$).
\end{itemize}	
\end{proof}

In the next lemma, the nonemptyness of certain subsets of $K_\tau^{a+}$ will be related to the condition $0 \notin {\rm cl}_\tau({\rm conv}(B_K))$. 
Later, in Theorem \ref{cor:0notinclSK}, we will be able to characterize the algebraic solidness of $K^{+}_\tau$ and $K^{a+}_\tau$, respectively.

\begin{lemma} \label{lem:0notinclSK}
Assume that $E$ is a real linear space, and $K$ is a nontrivial cone with normlike-base $B_{K}$. Define $S_K := {\rm conv}(B_K)$. Then, the following assertions hold:
	\begin{itemize}
		\item[$1^\circ$]  Let $(E, \tau)$ be a real topological linear space. 
		Then, \\$K_\tau^{a+} \cap (E^* \times \mathbb{P}) \neq \emptyset \Longrightarrow  0 \notin {\rm cl}_\tau\,S_K$. 
		\item[$2^\circ$]
		$K^{a+} \cap (E' \times \mathbb{P}) \neq \emptyset \Longrightarrow  0 \notin {\rm cl}_{\tau_c}\,S_K$, and if $S_K$ is relatively solid, then \\
		$K^{a+} \cap (E' \times \mathbb{P}) \neq \emptyset \Longrightarrow  0 \notin {\rm acl}\,S_K$.
		\item[$3^\circ$] Let $(E, \mathcal{F}, \tau)$ be a real locally convex space.	Then, \\$K_\tau^{a\#} \cap (E^* \times \mathbb{P})\neq \emptyset \iff K_\tau^{a+} \cap (E^* \times \mathbb{P}) \neq \emptyset \iff 0 \notin {\rm cl}_\tau\,S_K$.   
	\end{itemize}
\end{lemma}

\begin{proof}
\begin{itemize}
	\item[$1^\circ$] Assume that $(x^*, \alpha) \in K_\tau^{a+}$ with $\alpha > 0$. Hence, for any $y \in B_K$, we have 
	$x^*(y) \geq \alpha > 0 = x^*(0)$. Further, we get $x^*(y) \geq \alpha > x^*(0)$ for all $y \in {\rm cl}_{\tau}\,S_K$. This shows that $0 \notin {\rm cl}_{\tau}\,S_K$.
	\item[$2^\circ$] Directly follows by assertion $1^\circ$ by applying it for the convex core topology $\tau := \tau_c$. Notice that ${\rm cl}_{\tau_c}\,S_K = {\rm acl}\,S_K$ if the convex set $S_K$ is assumed to be relatively solid.
	\item[$3^\circ$] Since $K_\tau^{a\#} \subseteq K_\tau^{a+}$ we get $K_\tau^{a\#} \cap (E^* \times \mathbb{P}) \neq \emptyset \Longrightarrow K_\tau^{a+} \cap (E^* \times \mathbb{P})\neq \emptyset$. By Lemma \ref{lem:KaSharphasPosElements} ($1^\circ$) we also get the reverse implication. 
	By assertion $1^\circ$ we further have $K_\tau^{a+} \cap (E^* \times \mathbb{P}) \neq \emptyset \Longrightarrow  0 \notin {\rm cl}_\tau\,S_K$. Now, assume that  $0 \notin {\rm cl}_\tau\,S_K$. Since $\{0\}$ and ${\rm cl}_\tau\,S_K$ are nonempty, convex sets, ${\rm cl}_\tau\,S_K$ is $\tau$-closed and $\{0\}$ is $\tau$-compact, by the separation result in Jahn \cite[Th. 3.20]{Jahn2011} there are $x^* \in E^* \setminus \{0\}$ and $\alpha \in \mathbb{R}$ such that 
	$
	0 = x^*(0) < \alpha \leq x^*(x)$ for all $x \in {\rm cl}_\tau\,S_K \supseteq B_K.
	$
	This means that $(x^*, \alpha) \in K_\tau^{a+} \cap (E^* \times \mathbb{P})$. 
\end{itemize}	
\end{proof}

Using our previous results, we can state conditions for the nonemptyness of $K_\tau^{\#}$, ${\rm cor}\, K^{+}_\tau$ and ${\rm cor}\, K^{a+}_\tau$ for a nontrivial (not necessarily convex and closed) cone $K$.

\begin{theorem} \label{cor:0notinclSK}
Assume that $E$ is a real linear space, and $K$ is a nontrivial cone with normlike-base $B_{K}$. Define $S_K := {\rm conv}(B_K)$. Then, the following assertions hold:
\begin{itemize}
    \item[$1^\circ$] If $0 \notin S_{K}$,
    then $K$ is pointed.
    \item[$2^\circ$] If $(E, \mathcal{F}, \tau)$ is a real locally convex space, then \\$0 \notin {\rm cl}_\tau\,S_K \Longrightarrow K_\tau^{\#} \neq \emptyset$. 
    \item[$3^\circ$] If $(E, \tau)$ is a real topological linear space, $B_K$ is $\tau$-compact, then \\$K_\tau^{\#} \neq \emptyset \Longrightarrow 0 \notin {\rm cl}_\tau\,S_K$.
    \item[$4^\circ$] If $(E, \mathcal{F}, \tau)$ is a real separated locally convex space, and $B_K$ is $\tau$-compact, then \\${\rm cor}\, K^{+}_\tau \neq \emptyset \iff {\rm cor}\, K^{a+}_\tau \neq \emptyset \iff 0 \notin {\rm cl}_\tau\,S_K$.
\end{itemize}
\end{theorem}

\begin{proof}
\begin{itemize}
    \item[$1^\circ$]  Assuming $K$ is not pointed, i.e., $0 \neq x \in \ell(K) = K \cap (-K)$, we find $0 \neq \bar x \in B_K \cap \ell(K)$, hence $-\bar x \in B_K \cap \ell(K)$. This shows that $0 \in {\rm conv}(B_{K}) = S_{K}$. 
    \item[$2^\circ$] 
    $K_\tau^{a\#} \neq \emptyset$ implies $K_\tau^{\#} \neq \emptyset$, hence $0 \notin {\rm cl}_\tau\,S_K$ implies $K_\tau^{\#} \neq \emptyset$ in view of Lemma \ref{lem:0notinclSK} ($3^\circ$). 
    \item[$3^\circ$] Combining Lemma \ref{lem:KaPlushasPosElement} ($2^\circ$) and Lemma \ref{lem:0notinclSK} ($1^\circ$), we get directly this result. 
     \item[$4^\circ$] By Theorems \ref{th:cor_dual_cone_compact} and  \ref{th:cor_augmented_dual_cone_compact}, we have ${\rm cor}\, K^{+}_\tau = K_\tau^{\#}$ and ${\rm cor}\, K^{a+}_\tau = K_\tau^{a\#} \cap (E^* \times \mathbb{P})$ if $B_K$ is $\tau$-compact. Furthermore, Lemma \ref{lem:0notinclSK} ($3^\circ$) shows that ${\rm cor}\, K^{a+}_\tau \neq \emptyset \iff 0 \notin {\rm cl}_\tau\,S_K$ while $2^\circ$ and $3^\circ$ yield ${\rm cor}\, K^{+}_\tau \neq \emptyset \iff 0 \notin {\rm cl}_\tau\,S_K$.
\end{itemize}
    
\end{proof}

\section{Bishop-Phelps Cones in Normed Spaces} \label{sec:bp_cones}

Let $(E, ||\cdot||)$ be a real normed space. For any $x^* \in E^*$, we consider the well-known formulation of a Bishop-Phelps cone \cite{BP1962} given by
$
C(x^*) := \{ x \in E \mid x^*(x) \geq ||x||\}.
$
Usually one assumes that $||x^*||_* \geq 1$ (see Ha and Jahn \cite[Rem. 2.1]{HaJahn2021}) where $||\cdot||_*$ is the dual norm of $||\cdot||$. The class of Bishop-Phelps cones owns a lot of useful properties and has interesting applications in vector optimization (see, e.g., Eichfelder \cite{Eichfelder2014}, Eichfelder and Kasimbeyli \cite{EichfelderKasimbeyli2014}, Ha \cite{Ha2022}, Ha and Jahn \cite{HaJahn2017, HaJahn2021}, Jahn \cite{Jahn2009}, \cite[p. 159-160]{Jahn2011},  Kasimbeyli \cite{Kasimbeyli2010}, Kasimbeyli and Kasimbeyli \cite{KasKas17}). Next, we like to recall some properties of Bishop-Phelps cones. 
Clearly, $C(x^*)$ is a closed, pointed, convex cone. Moreover, if $||x^*||_* > 1$, then $C(x^*)$ is nontrivial; if $||x^*||_* < 1$, then $C(x^*) = \{0\}$; if $||x^*||_* = 1$, then $C(x^*)$ coincides with
$C_=(x^*) := \{ x \in E \mid x^*(x) = ||x||\}$ (i.e., $C(x^*)$ is a so-called Bishop-Phelps cone given by an equation, see \cite{HaJahn2021}), and if further $E$ is a real reflexive Banach space, then $C(x^*)\, (= C_=(x^*))$ is nontrivial.
Furthermore, if $E$ is a real Banach space, and $C(x^*)$ is nontrivial, then $||x^*||_* = 1 \iff C(x^*) = C_=(x^*)$.
Consider the following subset of $C(x^*)$ given by
$
C_>(x^*) := \{ x \in E \mid x^*(x) > ||x||\}.
$
Clearly, $C(x^*) = C_>(x^*)\; \dot{\cup}\; C_=(x^*)$. It is known that $C_>(x^*) \subseteq {\rm int}\, C(x^*)$, and $C_>(x^*) \neq \emptyset \iff ||x^*||_* > 1$. In the case $||x^*||_* > 1$, we further have ${\rm int}\, C(x^*) = {\rm cor}\, C(x^*) = C_>(x^*) \neq \emptyset$. If $E$ is a real strictly convex Banach space and $||x^*||_* = 1$, then ${\rm int}\, C(x^*) = {\rm int}\, C_=(x^*) = C_>(x^*) = \emptyset$ (see \cite[Prop. 4.1]{HaJahn2021}). However, if $E$ is a not strictly convex Banach space and $||x^*||_* = 1$, then ${\rm int}\, C_=(x^*) \neq \emptyset = C_>(x^*)$ may happen (see \cite[Sec. 4]{HaJahn2021}).

An interesting observation (see also Jahn \cite{Jahn2022}, \cite{Jahn2023}) is that the zero lower-level set of the (separation) function $\varphi_{x^*, \alpha}$ (with $\alpha > 0$), which is defined in \eqref{f12} by
\begin{equation}
    \label{sep_Kasimebyli}
   \varphi_{x^*, \alpha}(x)  = x^*(x) + \alpha ||x|| \quad \mbox{for all } x \in E, 
\end{equation}
is actually a Bishop-Phelps cone $C(-\frac{x^*}{\alpha})$, i.e.,
$$C_{x^*, \alpha}^{\leq} = \{x \in E \mid \varphi_{x^*, \alpha}(x) = x^*(x) + \alpha ||x|| \leq 0\} = C_{\frac{x^*}{\alpha}, 1}^{\leq} = - C(\frac{x^*}{\alpha}) = C(-\frac{x^*}{\alpha}).$$
Thus, the properties of the zero lower-level set $C_{x^*, \alpha}^{\leq}$ follow directly from the properties of the Bishop-Phelps cone $C(-\frac{x^*}{\alpha})$. 
Following some ideas by Ha \cite[Sec. 3.2]{Ha2022} and Jahn \cite[Ex. 2.1]{Jahn2022}, \cite[Rem. 2.1]{Jahn2023} one could consider the (separation) function $\varphi_{x^*}$ associated to the Bishop-Phelps cone $C(x^*)$ defined by
$$
\varphi_{x^*}(x)  := x^*(x) + ||x|| \quad \mbox{for all } x \in E.
$$
Therefore, functions $\varphi_{x^*, \alpha}$ in \eqref{sep_Kasimebyli} with  $x^* \in E^*$ and $\alpha > 0$ (respectively, $\varphi_{x^*}$) are also known as Bishop-Phelps functionals. 

For a given nontrivial cone $K \subseteq E$, let us define the following sets:
\begin{align*}
K^{BP+}_\tau &:= \{x^* \in E^* \mid K \subseteq C(x^*)\}, &
K^{BP\#}_\tau &:= \{x^* \in E^* \mid K \setminus \{0\} \subseteq C_>(x^*)\},\\
K^{BP\circ}_\tau &:= \{x^* \in E^* \mid {\rm icor}\, K \subseteq C_>(x^*)\},
& K^{BP\&}_\tau &:= \{x^* \in E^* \mid K \setminus \ell(K) \subseteq C_>(x^*)\}.
\end{align*}
Notice that $x^* \in K^{BP+}_\tau$ implies  $\{0\} \neq K \subseteq C(x^*)$, and so $||x^*||_* \geq 1$.  Moreover, $C(x^*)$ is $\tau$-solid (and so solid) and $||x^*||_* > 1$ (since $C_>(x^*) \neq \emptyset$) if
$x^* \in K^{BP\#}_\tau$; or if $x^* \in K^{BP\circ}_\tau$ and ${\rm icor}\, K \neq \emptyset$; or if $x^* \in K^{BP\&}_\tau$ and $K \neq \ell(K)$. 

Assume that $K$ is a nontrivial, pointed cone in $E$. A Bishop-Phelps cone $C(x^*)$ with $x^* \in K^{BP\#}_\tau$ is a so-called dilating cone for $K$ since $C(x^*)$ is convex cone with $K \setminus \{0\} \subseteq C_>(x^*) = {\rm cor}\, C(x^*) = {\rm int}\, C(x^*).$
It is known that dilating cones play an important role in vector optimization (see, e.g., Durea \cite{Durea2023}, G\"unther, Khazayel and Tammer \cite{Khazayel2021b}, Guti\'{e}rrez, Huerga and Novo \cite{GutHuergaNovo2018}, Henig \cite{Henig2}, Huerga, Jadamba and Sama \cite{HuergaJadambaSama2019},  Huerga, Khan and Sama \cite{HuergaKhanSama2019}, Khan and Sama \cite{KhanSama2013}, and Khan, Tammer and Z\u{a}linescu \cite[Sec. 2.4]{Khanetal2015}).
The mentioned dilating cone property also means 
$-K \setminus \{0\} \subseteq C_{x^*, 1}^{<}  = {\rm cor}\, C_{x^*, 1}^{\leq} = {\rm int}\, C_{x^*, 1}^{\leq}.$

In Section \ref{sec:cone_separation_lcs}, we will derive strict cone separation theorems.
More precisely, we state conditions such that the cone $-K$ and a nontrivial cone $A \subseteq E$ can be strictly separated (by a convex cone of Bishop-Phelps type).  For the choice $x^* \in K^{BP\#}_\tau$ one already has 
$-K \setminus \{0\}  \subseteq {\rm int}\, C_{x^*, 1}^{\leq} = \{x \in E \mid \varphi_{x^*}(x) = x^*(x) + ||x|| < 0\}.$
For strict cone separation of $-K$ and $A$, we have further to ensure that
$A\setminus \{0\} \subseteq E \setminus C_{x^*, 1}^{\leq} = \{x \in E \mid \varphi_{x^*}(x) = x^*(x) + ||x|| > 0\}.$
Notice that the following assertions are equivalent:
\begin{itemize}
    \item $-K$ and $A$ are strictly separated by the zero lower-level set of $\varphi_{x^*, \alpha}$, namely the set $C_{x^*, \alpha}^{\leq}\; = C_{\frac{x^*}{\alpha}, 1}^{\leq}$, for some $(x^*, \alpha) \in K^{a\#}_\tau \cap (E^* \times \mathbb{P})$.
    \item $-K$ and $A$ are strictly separated by the zero lower-level set of $\varphi_{x^*}$, namely the set $C_{x^*, 1}^{\leq}$, for some $x^* \in K^{BP\#}_\tau$.
    \item $-K \setminus \{0\} \subseteq {\rm int}\, C(y^*)$ and $A \setminus \{0\} \subseteq E \setminus C(y^*)$ for some $y^* \in -K^{BP\#}_\tau$.
    \item $K \setminus \{0\} \subseteq {\rm int}\, C(x^*)$ and $-A \setminus \{0\} \subseteq E \setminus C(x^*)$ for some $x^* \in K^{BP\#}_\tau$.
\end{itemize}
In the context of (weak, proper, strict) separation of two cones by the zero lower-level set of $\varphi_{x^*, \alpha}$, respectively, $\varphi_{x^*}$ , similar relationships between $K^{a+}_\tau \cap (E^* \times \mathbb{P})$ and $K^{BP+}_\tau$, respectively, $K^{a\circ}_\tau \cap (E^* \times \mathbb{P})$ and $K^{BP\circ}_\tau$, respectively, $K^{a\&}_\tau \cap (E^* \times \mathbb{P})$ and  $K^{BP\&}_\tau$ hold true.

In our presented separation approach, the convex sets $S_{-K} := {\rm conv}(-B_{K}) = {\rm conv}(B_{-K})$ and $S_{A}^0 := {\rm conv}(B_{A} \cup \{0\})$ for norm-bases (or more general norm-like bases) $B_{K}$ and $B_{A}$ of $K$ and $A$, as well as the condition 
\begin{equation}
\label{eq:-clSAcapclSK=empty}
({\rm cl}\,S_{A}^0) \cap  ({\rm cl}\, S_{-K}) = \emptyset
\end{equation}
will be of special interest. Figure \ref{fig:cone_separation} shows an example in a real normed space ($E, ||\cdot||_2$) where $||\cdot||_2$ denotes the Euclidean norm, $-K$ and $A$ are nontrivial, closed, pointed, solid cones that satisfy the conditions $A \cap (-K \setminus \{0\}) = \emptyset$, ${\rm cl}\,S_{A}^0 = S_{A}^0$ and  $0 \notin {\rm cl}\, S_{-K} = S_{-K}$.  
The separation condition  \eqref{eq:-clSAcapclSK=empty} is only valid in the left image of Figure \ref{fig:cone_separation} where $-K$ is convex and $A$ is nonconvex. One can also see that $-K$ and $A$ are strictly separated by (the boundary of) the Bishop-Phelps cone $C(y^*)$. To be precise, we have in this left image,
\begin{align*}
-K & \subseteq C(y^*) \quad \mbox{ and }\quad A \subseteq \mathbb{R}^2 \setminus {\rm int}\, C(y^*),\\
-K & \setminus \{0\} \subseteq {\rm int}\, C(y^*)\quad \mbox{ and } \quad A \setminus \{0\} \subseteq \mathbb{R}^2 \setminus C(y^*)
\end{align*}
for some $y^* \in -K^\#_\tau$. Notice that $x^* = - y^* \in K^{BP\#}_\tau$ and $C(y^*) = C(-x^*) = C_{x^*, 1}^{\leq}$ as well as ${\rm int}\, C(y^*) = {\rm int}\, C_{x^*, 1}^{\leq}$.

\begin{figure}[h!]
\resizebox{1\hsize}{!}{
\input{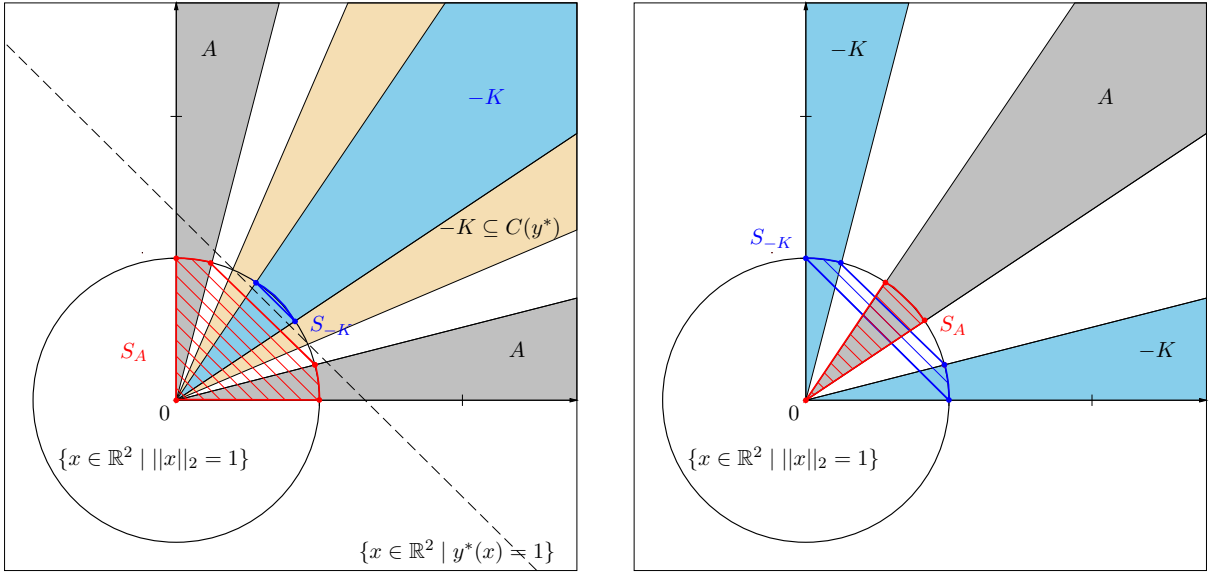}
\label{fig:cone_separation}
}
\caption{Cone Separation of two nontrivial, closed, pointed, solid cones $-K$ and $A$ that satisfy $A \cap (-K \setminus \{0\}) = \emptyset$, ${\rm cl}\,S_{A}^0 = S_{A}^0$ and $0 \notin {\rm cl}\, S_{-K} = S_{-K}$ in the real normed space $(\mathbb{R}^2, ||\cdot||_2)$: \\
(left image) $-K$ is convex, $A$ is nonconvex, \eqref{eq:-clSAcapclSK=empty} is valid;\\
 (right image) $-K$ is nonconvex, $A$ is convex, \eqref{eq:-clSAcapclSK=empty} is not valid.}
\end{figure}

For deriving our separation theorems in the next section, we follow basically and we extend the approach by Kasimbeyli \cite{Kasimbeyli2010} which is based on the separating function $\varphi_{x^*, \alpha}$ defined by \eqref{sep_Kasimebyli},
where $(x^*, \alpha)$ is taken from the augmented dual cone $K^{a+}_\tau$. Our generalized separating function given in \eqref{f31}, namely 
$\varphi_{x', \alpha}$ for some $(x', \alpha) \in K^{a+}$, will always involve a linear function $x' \in K^+$, a seminorm $\psi$ and augmentation parameter $\alpha \geq 0$. 

\section{Cone Separation Theorems in (Topological) Linear Spaces} \label{sec:cone_separation_lcs}

In this section, assume that $E$ is a real linear space, and $\psi: E \to \mathbb{R}$ is a seminorm. We like to present new nonlinear cone separation theorems in real linear spaces, real topological linear spaces, and real locally convex spaces, respectively. The main tool in the proof of our separation results (namely Theorems \ref{th:sep_main}, \ref{cor:sep_main_int}, \ref{th:sep_main_cor},  \ref{th:sep_main_int}, \ref{th:sep_main_top}) is the application of classical separation theorems (compare Jahn \cite[Th. 3.14, 3.16, 3.20]{Jahn2011}, and Mill\'{a}n and Roshchina \cite[Prop. 4.12]{Roshchina2021}).

\subsection{Weak Cone Separation Theorems}

In the first part of the section, we are going to study the case of weak cone separation.

\begin{theorem}\label{th:sep_main}
Assume that $E$ is a real linear space, and $K, A \subseteq E$ are nontrivial cones with normlike-bases $B_{K}$ and $B_{A}$. Define $S_{-K} := {\rm conv}(-B_{K})$ and $S_{A}^0 := {\rm conv}(B_{A} \cup \{0\})$.
If 
one of the following conditions hold:
\begin{itemize}
\item $S_{-K}$ is solid and $S_{A}^0 \cap {\rm cor}\,S_{-K} = \emptyset$, 
\item $S_{A}^0$ is solid and $({\rm cor}\,S_{A}^0) \cap S_{-K} = \emptyset$, 
\item $S_{-K}$ and $S_{A}^0$ are relatively solid and $({\rm icor}\,S_{A}^0) \cap ({\rm icor}\,S_{-K}) = \emptyset$, 
\end{itemize}
then there exists $(x', \alpha) \in K^{a+}$ with $x' \neq 0$  such that
\begin{equation}
\label{eq:sep_alg_<=}
x'(a) + \alpha \psi(a) \geq  0 \geq x'(k) + \alpha \psi(k)   \quad \mbox{for all } a \in A \mbox{ and }k \in -K.
\end{equation}
\end{theorem}

\begin{proof}
By the linear separation result in Jahn \cite[Th. 3.14]{Jahn2011} (respectively, in Mill\'{a}n and Roshchina \cite[Prop. 4.12]{Roshchina2021} for the intrinsic case) there are $x' \in E' \setminus \{0\}$ and $\delta \in \mathbb{R}$ such that
\begin{equation}
\label{eq:linSepConSetsSAandSK}
x'(a) \geq \delta \geq x'(k)   \quad \mbox{for all } a \in S_{A}^0 \mbox{ and }k \in S_{-K}.
\end{equation}	
Since $0 \in S_{A}^0$ we get $\delta \leq 0$. Then, by \eqref{eq:linSepConSetsSAandSK} we have $x'(k) \leq 0$ for all $k \in S_{-K} \supseteq -B_K$, hence $x'(k) \geq 0$ for all $k \in K$, i.e., $x' \in K^+ \setminus \{0\}$. Define $\alpha := - \delta\; (\geq 0)$.
From the first inequality in \eqref{eq:linSepConSetsSAandSK} and for elements $a \in B_{A}$ (i.e., $\psi(a) = 1$), we obtain
$
x'(a) + \alpha \psi(a)  = x'(a) - \delta \geq 0$  for all $a \in B_{A}$.
Because $B_A$ is a base for $A$, and $x'$ and $\psi$ are positively homogeneous, it follows $
 x'(a) + \alpha \psi(a)  \geq 0$ for all $a \in A$.   
From the second inequality in \eqref{eq:linSepConSetsSAandSK} and for elements $k \in -B_{K}$ (i.e., $\psi(k) = 1$), we get
$
0 \geq x'(k) - \delta = x'(k) + \alpha \psi(k)$   for all $k \in -B_{K}$.
Since $-B_K = B_{-K}$ is a base for $-K$, and $x'$ and $\psi$ are positively homogeneous, it follows
$
0 \geq x'(k) + \alpha \psi(k)$ for all $k \in -K$.   
We conclude that $(x', \alpha) \in K^{a+}$ and \eqref{eq:sep_alg_<=} hold true.
\end{proof}

\begin{theorem}\label{cor:sep_main_int}
	Assume that $(E, \tau)$ is a real topological linear space, and $K, A \subseteq E$ are nontrivial cones with normlike-bases $B_{K}$ and $B_{A}$. Define $S_{-K} := {\rm conv}(-B_{K})$ and $S_{A}^0 := {\rm conv}(B_{A} \cup \{0\})$.
	If one of the following conditions is satisfied:
	\begin{itemize}
		\item $S_{-K}$ is $\tau$-solid and $S_{A}^0 \cap {\rm int}_\tau \,S_{-K} = \emptyset$, 
		\item $S_{A}^0$ is $\tau$-solid and $({\rm int}_\tau\,S_{A}^0) \cap S_{-K} = \emptyset$, 
	\end{itemize}
	then there exists $(x^*, \alpha) \in K_\tau^{a+}$ with $x^* \neq 0$  such that
	\begin{equation}
	\label{eq:sep_top_<=}
	x^*(a) + \alpha \psi(a) \geq  0 \geq x^*(k) + \alpha \psi(k)   \quad \mbox{for all } a \in A \mbox{ and }k \in -K.
	\end{equation}
\end{theorem}

\begin{proof}
	The proof is similar to the proof of Theorem \ref{th:sep_main} by using the linear separation result in Jahn \cite[Th. 3.16]{Jahn2011}.
\end{proof}

\subsection{Proper Cone Separation Theorems}

In the second part of the section, we concentrate on the case of proper cone separation. 

\begin{theorem} \label{th:sep_main_cor}
	Assume that $E$ is a real linear space, $K, A \subseteq E$ are nontrivial cones with normlike-bases $B_{K}$ and $B_{A}$, and $K$ is solid. Define $S_{-K} := {\rm conv}(-B_{K})$ and $S_{A}^0 := {\rm conv}(B_{A} \cup \{0\})$. If $S_{-K}$ is solid and
	$
	S_{A}^0\cap {\rm cor}\,S_{-K} = \emptyset,
	$
	then there exists $(x', \alpha) \in K^{a\circ}$ such that
	\begin{align}
	x'(a) + \alpha \psi(a) & \geq  0 \geq x'(k) + \alpha \psi(k)   \quad \mbox{for all } a \in A \mbox{ and }k \in -K, \label{eq:separation_nonlinear_cor1}\\
	x'(a) + \alpha \psi(a) & \geq  0 > x'(k) + \alpha \psi(k)   \quad \mbox{for all } a \in A \mbox{ and }k \in {\rm cor}(-K).\label{eq:separation_nonlinear_cor2}
	\end{align}
	In particular, we have $A \cap {\rm cor}(-K) = \emptyset$ and $0 \notin {\rm cor}\,S_{-K}$.
\end{theorem}

\begin{proof}
	By the linear separation result in Jahn \cite[Th. 3.14]{Jahn2011} (applied for the two nonempty, convex sets $S_{-K}$ and $S_{A}^0$) there are $x' \in E' \setminus \{0\}$ and $\delta \in \mathbb{R}$ such that
	\begin{align}
	\label{eq:lineSepConSetSAandSK1}
	x'(a) & \geq \delta \geq x'(k)   \quad \mbox{for all } a \in S_{A}^0 \mbox{ and }k \in S_{-K},\\
	\label{eq:lineSepConSetSAandSK2}
	x'(a) & \geq \delta > x'(k)   \quad \mbox{for all } a \in S_{A}^0 \mbox{ and }k \in {\rm cor}\,S_{-K}.
	\end{align}	
	Since $0 \in S_{A}^0$ we get $\delta \leq 0$, hence $x' \in K^+ \setminus \{0\}$. 	Define $\alpha := - \delta\; (\geq 0)$.
	From the first inequality in \eqref{eq:lineSepConSetSAandSK1} and for $a \in B_{A}$ (i.e., $\psi(a) = 1$), we conclude
	$
	x'(a) + \alpha \psi(a) = x'(a) - \delta \geq 0$ for all $a \in B_{A}$,
	hence
	$
	x'(a) + \alpha \psi(a) \geq 0$ for all $a \in A$.
	Further, from the second inequality in \eqref{eq:lineSepConSetSAandSK2}, it follows 
	$
	0 > x'(k) - \delta =  x'(k) + \alpha$ for all $k \in {\rm cor}\,S_{-K}$.
	Since $\psi(k) \in [0,1]$ for $k \in {\rm cor}\,S_{-K} \subseteq {\rm conv}(-B_K)$, we get
	$
	0 > x'(k) + \alpha \psi(k)$ for all $k \in {\rm cor}\,S_{-K}$.
	Due to the fact that $x'$ and $\psi$ are positively homogeneous, we derive
	$
	0 > x'(k) + \alpha \psi(k)$  for all $k \in  \mathbb{P} \cdot {\rm cor}\,S_{-K}$.
	By Lemma \ref{lem:K=convB} ($2^\circ$) we have ${\rm cor}(-K) \subseteq \mathbb{P} \cdot {\rm cor}\,S_{-K}$. Thus, \eqref{eq:separation_nonlinear_cor2} is valid. Taking into account that ${\rm cor}(-K) = - {\rm cor}\, K$, we get 
	$
	x'(k) - \alpha \psi(k) > 0$ for all $k \in  {\rm cor}\, K$,
	i.e., $(x', \alpha) \in K^{a\circ}$ holds true. 
	Using the second inequality in \eqref{eq:lineSepConSetSAandSK1} and similar ideas as above, one gets 
	$
	0 \geq x'(k) + \alpha \psi(k)$ for all $k \in  \mathbb{R}_+ \cdot S_{-K} \supseteq \mathbb{R}_+ \cdot (-B_K) = -K$,
	which shows that \eqref{eq:separation_nonlinear_cor1} is valid.
\end{proof}

\begin{theorem} \label{th:sep_main_int}
	Assume that $(E, \tau)$ is a topological linear space, $K, A \subseteq E$ are nontrivial cones with normlike-bases $B_{K}$ and $B_{A}$, and $K$ is $\tau$-solid. Define $S_{-K} := {\rm conv}(-B_{K})$ and $S_{A}^0 := {\rm conv}(B_{A} \cup \{0\})$. If $S_{-K}$ is $\tau$-solid and
	$
	S_{A}^0 \cap {\rm int}_\tau \,S_{-K} = \emptyset,
	$
	then there exists $(x^*, \alpha) \in K_\tau^{a\circ}$ such that
	\begin{align*}
	x^*(a) + \alpha \psi(a) & \geq  0 \geq x^*(k) + \alpha \psi(k)   \quad \mbox{for all } a \in A \mbox{ and }k \in -K,\\
	x^*(a) + \alpha \psi(a) & \geq  0 > x^*(k) + \alpha \psi(k)   \quad \mbox{for all } a \in A \mbox{ and }k \in {\rm int}_\tau(-K).
	\end{align*}
	In particular, we have $A \cap {\rm int}_\tau(-K) = \emptyset$ and $0 \notin {\rm int}_\tau \,S_{-K}$.  
\end{theorem}

\begin{proof}
	The proof is similar to the proof of Theorem \ref{th:sep_main_cor} by using the  separation result in Jahn \cite[Th. 3.16]{Jahn2011}.
	Notice that ${\rm cor}\,S_{-K} = {\rm int}_\tau \,S_{-K}$ holds true for the $\tau$-solid, convex set $S_{-K}$ in $(E, \tau)$, hence ${\rm int}_\tau(-K) \subseteq {\rm cor}(-K) \subseteq \mathbb{P} \cdot {\rm cor}\,S_{-K} = \mathbb{P} \cdot {\rm int}_\tau \,S_{-K}$.
\end{proof}

Let us consider for some moment the convex case. 

\begin{theorem} 
	\label{th:iffCorSAandSK}
	Assume that $E$ is a real linear space, $K, A \subseteq E$ are nontrivial, convex cones with normlike-bases $B_{K}$ and $B_{A}$, and $K$ is solid.  Define $S_{-K} := {\rm conv}(-B_{K})$ and $S_{A}^0 := {\rm conv}(B_{A} \cup \{0\})$. 
	Suppose that $S_{-K}$ is solid.	
	Then, the following assertions are equivalent:
	\begin{itemize}
		\item[$1^\circ$] $S_{A}^0 \cap {\rm cor}\,S_{-K} = \emptyset$.		
		\item[$2^\circ$] $A \cap {\rm cor}(-K) = \emptyset$.
		\item[$3^\circ$] There exists $x' \in K^+ \setminus \{0\}$ such that
		$
		x'(a)  \geq  0 > x'(k)$ for all $a \in A$ and $k \in {\rm cor}(-K)$.
		\item[$4^\circ$] There exists $(x', \alpha) \in K^{a\circ}$ such that $x'(a) + \alpha \psi(a)  \geq  0 > x'(k) + \alpha \psi(k)$ for all $a \in A$ and $k \in {\rm cor}(-K)$.
	\end{itemize}
\end{theorem}

\begin{proof}
	Obviously, $4^\circ \Longrightarrow 2^\circ$ is valid. Moreover, $2^\circ \Longrightarrow 1^\circ$ follows by the fact that 
	$
	A \cap {\rm cor}(-K) = A \cap (\mathbb{P} \cdot {\rm cor}({\rm conv}(-B_K)) \supseteq S_{A}^0 \cap {\rm cor}\,S_{-K}
	$
	taking into account the convexity of $-K$ and $A$, and Lemma \ref{lem:K=convB} ($2^\circ$, $3^\circ$).	
	Theorem \ref{th:sep_main_cor} provides the implication $1^\circ \Longrightarrow 4^\circ$. The equivalence $2^\circ \iff 3^\circ$ is a direct consequence of the linear separation result in \cite[Th. 3.14]{Jahn2011} (under the convexity of $-K$ and $A$).
\end{proof}

\begin{remark}
	Theorem \ref{th:sep_main_cor} shows that the condition $S_{A}^0 \cap {\rm cor}\,S_{-K} = \emptyset$ implies $A \cap {\rm cor}(-K) = \emptyset$ for not necessarily convex cones $K, A \subseteq E$. Moreover, Theorem \ref{th:iffCorSAandSK} shows that the reverse implication is valid under convexity assumptions on $K$ and $A$.
	Without the convexity assumptions, the condition $A \cap {\rm cor}(-K) = \emptyset$ does not imply the condition $S_{A}^0 \cap {\rm cor}\,S_{-K} = \emptyset$, as the following example shows.
\end{remark}

\begin{example} \label{ex:counterex_AcapcorK}
	Consider the normed space $(E = \mathbb{R}^2, ||\cdot||_2)$, where $||\cdot||_2$ denotes the Euclidean norm. Define $x^{(\lambda)} := (1- \lambda, \lambda)$ for all $\lambda \in [0,1]$ and put 
	$\Omega^1 := \mathbb{R}_+ \cdot [x^{(0)}, x^{(\frac{1}{5})}]$, $\Omega^2 := \mathbb{R}_+ \cdot [x^{(\frac{2}{5})}, x^{(\frac{3}{5})}]$ and $\Omega^3 := \mathbb{R}_+ \cdot [x^{(\frac{4}{5})}, x^{(1)}]$.
	Moreover, consider two cones $K := -(\Omega^1 \cup \Omega^3)$ and $A := \Omega^2$. Notice that  $K$ and $A$ are nontrivial, pointed, solid cones, $A$ is convex but $K$ is nonconvex.
	Then, $S_{-K}$ is solid, $A \cap {\rm cor}(-K) \subseteq A \cap (-K \setminus \{0\}) = \emptyset$ but
	$
	S_{A}^0 \cap {\rm cor}\,S_{-K}\neq \emptyset.
	$
	The example is illustrated in the right image of Figure \ref{fig:cone_separation}. 
\end{example}

\begin{theorem} 
	Assume that $(E, \tau)$ is a real topological linear space, $K, A \subseteq E$ are nontrivial, convex cones with normlike-bases $B_{K}$ and $B_{A}$, and $K$ is $\tau$-solid. Define $S_{-K} := {\rm conv}(-B_{K})$ and $S_{A}^0 := {\rm conv}(B_{A} \cup \{0\})$. 
	Suppose that $S_{-K}$ is $\tau$-solid.	
	Then, the following assertions are equivalent:	
	\begin{itemize}
		\item[$1^\circ$] $S_{A}^0 \cap {\rm int}_\tau \,S_{-K} = \emptyset$.		
		\item[$2^\circ$] $A \cap {\rm int}_\tau(-K) = \emptyset$.
		\item[$3^\circ$] There exists $x^* \in K^+_\tau \setminus \{0\}$ such that $x^*(a)  \geq  0 > x^*(k)$ for all $a \in A$ and $k \in {\rm int}_\tau(-K)$.
		\item[$4^\circ$] There exists $(x^*, \alpha) \in K_\tau^{a\circ}$ such that $x^*(a) + \alpha \psi(a)  \geq  0 > x^*(k) + \alpha \psi(k)$ for all $a \in A$ and $k \in {\rm int}_\tau(-K)$.
	\end{itemize}
\end{theorem}

\begin{proof}
	The proof is similar to the proof of Theorem \ref{th:iffCorSAandSK}. Notice that ${\rm int}_\tau(-K) = {\rm cor}(-K)$ and ${\rm cor}\,S_{-K} = {\rm int}_\tau \,S_{-K}$ for $\tau$-solid, convex sets $-K$ and $S_{-K}$ in $(E, \tau)$.
\end{proof}

\subsection{Strict Cone Separation Theorems}

In the remaining part of the section, we are going to study the case of strict cone separation. 

\begin{theorem} \label{th:sep_main_top}
	Assume that $(E, \mathcal{F}, \tau)$ is a real locally convex space,  $K, A \subseteq E$ are nontrivial cones with normlike-bases $B_{K}$ and  $B_{A}$, and $K$ is pointed.
	Define  
	$S_{-K} := {\rm conv}(-B_{K})$ and $S_{A}^0 := {\rm conv}(B_{A} \cup \{0\})$. Suppose that one of the sets ${\rm cl}_\tau\, S_{-K}$ and ${\rm cl}_\tau\,S_{A}^0$ is $\tau$-compact.
	If 
	$
	({\rm cl}_\tau\,S_{A}^0) \cap ({\rm cl}_\tau\, S_{-K}) = \emptyset,
	$
	then there exists $(x^*, \alpha) \in K_\tau^{a\#} \cap (E^* \times \mathbb{P})$ such that
	\begin{align}
	x^*(a) + \alpha \psi(a) & \geq  0 \geq x^*(k) + \alpha \psi(k)   \quad \mbox{for all } a \in A \mbox{ and }k \in -K, \label{sep_non_top_1}\\
	x^*(a) + \alpha \psi(a) &  >  0 > x^*(k) + \alpha \psi(k)   \quad \mbox{for all } a \in A\setminus \{0\} \mbox{ and }k \in -K \setminus \{0\}. \label{sep_non_top_2}
	\end{align}
	In particular, we have $A \cap (-K \setminus \{0\}) = \emptyset$ and $0 \notin {\rm cl}_\tau\, S_{-K}$. 
\end{theorem}

\begin{proof}
    First, notice that ${\rm cl}_\tau\,S_{-K}$ and ${\rm cl}_\tau\,S_{A}^0$ are nonempty, $\tau$-closed, convex sets.  
	By the strict linear separation of convex sets in locally convex spaces (see Jahn \cite[Th. 3.20]{Jahn2011}) there are $x^* \in E^* \setminus \{0\}$ and $\gamma, \beta \in \mathbb{R}$ such that
	\begin{equation}
	\label{eq:linSepConSetSAandCLSK}
	x^*(a) \geq \beta > \gamma \geq x^*(k)   \quad \mbox{for all } a \in {\rm cl}_\tau\,S_{A}^0 \mbox{ and }k \in {\rm cl}_\tau\,S_{-K}.
	\end{equation}	
	Since $0 \in S_{A}^0$ we get $\beta \leq 0$. Then, by \eqref{eq:linSepConSetSAandCLSK} we have $x^*(k) < 0$ for all $k \in {\rm cl}_\tau\,S_{-K}  \supseteq -B_K$, hence $x^*(k) > 0$ for all $k \in K \setminus \{0\}$, i.e., $x^* \in K^\#_\tau$. Take some $\delta \in (\gamma, \beta) \subseteq (-\infty, 0)$. 
	Define $\alpha := - \delta\; (> 0)$.
	From the first inequality in \eqref{eq:linSepConSetSAandCLSK}, the fact that $\delta < \beta$, and for $a \in B_A$ (i.e., $\psi(a) = 1$), we conclude
	$
	x^*(a) + \alpha \psi(a) = x^*(a) - \delta  >  0$   for all $a \in B_{A}$.
	Because $B_A$ is a base for $A$, and $x^*$ and $\psi$ are positively homogeneous, it follows $x^*(a) + \alpha \psi(a) \geq 0$  for all $a \in A$, and $x^*(a) + \alpha \psi(a)  > 0$ for all $a \in A \setminus \{0\}$.
	From the third inequality in \eqref{eq:linSepConSetSAandCLSK}, the fact that $\delta > \gamma$, and for $k \in -B_{K}$ (i.e., $\psi(k) = 1$), we get
	$
	0 > x^*(k) - \delta = x^*(k) + \alpha \psi(k)$ for all $k \in -B_{K}$.
	Since $B_K$ is a base for $K$, and $x^*$ and $\psi$ are positively homogeneous, we get $0 \geq x^*(k) + \alpha \psi(k)$ for all $k \in -K$, and $0 > x^*(k) + \alpha \psi(k)$ for all $k \in -K \setminus \{0\}$.
	We conclude that both conditions \eqref{sep_non_top_1} and \eqref{sep_non_top_2} are valid. It is easy to see that we also derive $(x^*, \alpha) \in K^{a\#}_\tau \cap (E^* \times \mathbb{P})$.
\end{proof}

\begin{remark} Notice, due to $-B_K = B_{-K}$ we have $S_{-K} = {\rm conv}(-B_{K}) = {\rm conv}(B_{-K})$. 
Since, for any set $\Omega \subseteq E$ the equality $-{\rm cl}({\rm conv}(\Omega)) = {\rm cl}({\rm conv}(-\Omega))$ holds true (see Swartz \cite[Ch. 11, Ex. 6]{Swartz}), it is easy to check that 
$({\rm cl}\,S_{A}^0) \cap ({\rm cl}\, S_{-K}) = \emptyset$ if and only if 
$({\rm cl}\,S_{A}^0) \cap (-{\rm cl}\, S_K) = \emptyset$, 
where $S_{K} := {\rm conv}(B_{K})$. Thus, $({\rm cl}\,S_{A}^0) \cap ({\rm cl}\, S_{-K}) = \emptyset$ implies also $0 \notin {\rm cl}\, S_K$. In view of Theorem \ref{cor:0notinclSK} ($1^\circ$),
it is justified simply to assume the pointedness of the nontrivial cone $K$ in Theorem \ref{th:sep_main_top} and also in upcoming cone separation results.
\end{remark}

\begin{theorem}
	Assume that $E$ is a real linear space, $K, A \subseteq E$ are nontrivial cones with  normlike-bases $B_{K}$ and  $B_{A}$, and $K$ is pointed.  Define the sets 
	$S_{-K} := {\rm conv}(-B_{K})$ and $S_{A}^0 := {\rm conv}(B_{A} \cup \{0\})$.	
	If there exists $(x', \alpha) \in K^{a\#}$ such that
	\begin{align}\label{eq:sep_Ksetminus0_algebraic}
	x'(a) + \alpha \psi(a) &  \geq  0 > x'(k) + \alpha \psi(k)   \quad \mbox{for all } a \in A \mbox{ and }k \in -K \setminus \{0\},
	\end{align}
	then $S_{A}^0 \cap S_{-K} = \emptyset$.
\end{theorem}

\begin{proof} Assume that there exists $(x', \alpha) \in K^{a\#}$ such that \eqref{eq:sep_Ksetminus0_algebraic}. Then, we also have
	$
	x'(a) + \alpha \psi(a)  \geq  0 > x'(k) + \alpha \psi(k)$ for all $a \in B_A \cup \{0\}$ and $k \in -B_K$,
	hence $x'(a) \geq  -\alpha > x'(k)$  for all $a \in B_A \cup \{0\}$ and $k \in -B_K$. 
	By the convexity of (open) half spaces,
	$x'(a)   \geq -\alpha > x'(k)$ for all $a \in S_{A}^0$ and $k \in S_{-K}$, 
	i.e., $S_{A}^0 \cap S_{-K} = \emptyset$.
\end{proof}

\begin{theorem} \label{th:sep_main_top_back}
	Assume that $(E, \tau)$ is a real topological linear space, $K, A \subseteq E$ are nontrivial cones with  normlike-bases $B_{K}$ and  $B_{A}$, and $K$ is pointed. Define 
	$S_{-K} := {\rm conv}(-B_{K})$ and $S_{A}^0 := {\rm conv}(B_{A} \cup \{0\})$. Assume that there exists $(x^*, \alpha) \in K_\tau^{a\#}$ with
	\begin{align}\label{eq:sep_Ksetminus0_topologic}
	x^*(a) + \alpha \psi(a) &  \geq  0 > x^*(k) + \alpha \psi(k)   \quad \mbox{for all } a \in A \mbox{ and }k \in -K \setminus \{0\}.
	\end{align}
	Suppose further that one of the following conditions is fulfilled:	
	\begin{itemize}
		\item[$1^\circ$] $B_{K}$ is $\tau$-compact, 
		\item[$2^\circ$] ${\rm cl}_\tau\, S_{-K}$ is $\tau$-compact, $0 \notin {\rm cl}_\tau\, S_{-K}$ (e.g, if $\alpha > 0$ or $B_{K}$ is $\tau$-compact), $\psi$ is $\tau$-continuous, and $K$ is $\tau$-closed and convex.
	\end{itemize}
	Then,  $({\rm cl}_\tau\,S_{A}^0) \cap ({\rm cl}_\tau\, S_{-K}) = \emptyset$.
\end{theorem}

\begin{proof} Firstly, assume that $1^\circ$ is valid. 
	By our assumptions, there exists $(x^*, \alpha) \in K_\tau^{a\#}$ such that \eqref{eq:sep_Ksetminus0_topologic} holds true. This means in particular that 
	$x^*(a) + \alpha \psi(a)   \geq  0 > x^*(k) + \alpha \psi(k)$ for all $a \in B_A \cup \{0\}$ and all $k \in -B_{K} \subseteq -K \setminus \{0\}$.	
	Due to the fact that $B_A$ and $B_K$ are normlike-bases (i.e., $\psi(a) = \psi(k) = 1$ for all $a \in B_A, k \in -B_{K}$), we get 
	$
	x^*(a) \geq  -\alpha > x^*(k)$ for all $a \in B_A \cup \{0\}$  and all $k \in -B_{K}$.	
	Since $x^*$ is $\tau$-continuous and $B_{K}$ is $\tau$-compact, there is $\beta < 0$ such that
	$
	x^*(a) \geq  -\alpha > \beta \geq x^*(k)$  for all $a \in B_A \cup \{0\}$ and all $k \in -B_{K}$ 
	by a general version of the Weierstraß theorem (see \cite[Th. 3.26]{Jahn2011}). 
	By the convexity of closed half spaces and the $\tau$-continuity of $x^*$, we get $
	x^*(a)   \geq  -\alpha > \beta \geq x^*(k)$ for all $a \in {\rm cl}_\tau\, S_{A}^0$ and all $k \in {\rm cl}_\tau\, S_{-K}$.
	As a direct consequence we conclude that
	$({\rm cl}_\tau\, S_{A}^0) \cap ({\rm cl}_\tau\, S_{-K}) = \emptyset$.

	Secondly, assume that $2^\circ$ is valid. Since $K$ is convex and $\tau$-closed, we have ${\rm cl}_\tau\, S_{-K} = -{\rm cl}_\tau({\rm conv}(B_{K}))\subseteq -{\rm cl}_\tau\, K = -K $. Moreover, by $0 \notin {\rm cl}_\tau\, S_{K}$, we conclude that ${\rm cl}_\tau\, S_{-K} \subseteq -K \setminus \{0\}$.
	By our assumptions, there exists $(x^*, \alpha) \in K_\tau^{a\#}$ such that \eqref{eq:sep_Ksetminus0_topologic} holds true. In particular, we obtain $
	x^*(a) + \alpha \psi(a) \geq  0 > x^*(-k) + \alpha \psi(k)$ for all $a \in B_A \cup \{0\}$ and all $k \in {\rm cl}_\tau\, S_{-K}$.	
	Since $k \mapsto x^*(k) + \alpha \psi(k)$ is $\tau$-continuous and $ {\rm cl}_\tau\, S_{-K}$ is $\tau$-compact, there is $\beta < 0$ such that
	$
	x^*(a) + \alpha \psi(a)  \geq  0 > \beta \geq x^*(k) + \alpha \psi(k)$  for all $a \in B_A \cup \{0\}$ and $k \in  {\rm cl}_\tau\, S_{-K} \supseteq -B_K$
	by the Weierstraß theorem. Define $\gamma := \beta - \alpha$. Due to the fact that $B_A$ and $B_K$ are normlike-bases (i.e., $\psi(a) = \psi(k) = 1$ for all $a \in B_A, k \in -B_{K}$), we get 
	$
	x^*(a)  \geq  -\alpha > \gamma \geq x^*(k)$ for all $a \in B_A \cup \{0\}$ and all $k \in -B_K$. 
	By the convexity of closed half spaces and the $\tau$-continuity of $x^*$, it follows
	$
	x^*(a)  \geq  -\alpha > \gamma \geq x^*(k)$   for all $a \in {\rm cl}_\tau\, S_{A}^0$ and all $k \in {\rm cl}_\tau\, S_{-K}$. 
	As desired, also in the second case we conclude
	$({\rm cl}_\tau\,S_{A}^0) \cap ({\rm cl}_\tau\, S_{-K}) = \emptyset$.
\end{proof}

Next, we state our main strict cone separation theorem. 

\begin{theorem} \label{th:sep_main_top_iff}
	Assume that $(E, \mathcal{F}, \tau)$ is a real locally convex space, $K, A \subseteq E$ are nontrivial cones with normlike-bases $B_{K}$ and $B_{A}$, and $K$ is pointed. Define
	$S_{-K} := {\rm conv}(-B_{K})$ and $S_{A}^0 := {\rm conv}(B_{A} \cup \{0\})$. Consider the two assertions:	
	\begin{itemize}
	    \item[$1^\circ$] 
	    $
	    ({\rm cl}_\tau\,S_{A}^0) \cap ({\rm cl}_\tau\, S_{-K}) = \emptyset.
	    $
	    \item[$2^\circ$]  There exists $(x^*, \alpha) \in K_\tau^{a\#} \cap (E^* \times \mathbb{P})$ such that \eqref{sep_non_top_1} and \eqref{sep_non_top_2} are valid.
	\end{itemize}	
	If one of the sets ${\rm cl}_\tau\, S_{-K}$ and ${\rm cl}_\tau\,S_{A}^0$ is $\tau$-compact, then $1^\circ \Longrightarrow 2^\circ$.\\
	If either $B_{K}$ is $\tau$-compact or ${\rm cl}_\tau\, S_{-K}$ is $\tau$-compact, $\psi$ is $\tau$-continuous (e.g., $\psi \in \mathcal{F}$),  and $K$ is $\tau$-closed and convex, then $2^\circ \Longrightarrow 1^\circ$.
\end{theorem}

\begin{proof}
	Combining the results in Theorems \ref{th:sep_main_top} and \ref{th:sep_main_top_back} we get this result.
\end{proof}

\begin{remark} \label{rem:sepBoundaryA}
	Assume that $(E, \mathcal{F}, \tau)$ is a real locally convex space, and $A \subseteq E$ is a nontrivial, $\tau$-closed cone with normlike-base 
	$
	B_A = \{a \in A \mid \psi(a) = 1 \}.
	$
	Then, $\tilde{A} := {\rm bd}_{\tau}\,A$ is a nontrivial, $\tau$-closed cone with normlike-base
	$$
	B_{\tilde{A}} = \{a \in \tilde{A} \mid \psi(a) = 1 \} = \{a \in A \mid \psi(a) = 1 \} \cap  \tilde{A} = B_A \cap {\rm bd}_{\tau}\,A.
	$$
	Thus, in the separation results given in Theorems \ref{th:sep_main_top} and \ref{th:sep_main_top_back} one could replace $S_A^0$ by $S_{\partial A}^0 := S_{\tilde{A}}^0 = {\rm conv}((B_A \cap {\rm bd}_{\tau}\,A) \cup \{0\})$ and state the corresponding conclusions using ${\rm bd}_{\tau}\,A$ instead of $A$. By doing this, we get a similar result (see Corollary \ref{th:sep_main_top_iff_kasimbeyli}) as derived in the  paper by Kasimbeyli \cite[Th. 4.3]{Kasimbeyli2010}.
\end{remark}

\begin{corollary} \label{th:sep_main_top_iff_kasimbeyli}
	Assume that $(E, \mathcal{F}, \tau)$ is a real locally convex space, $K \subseteq E$ is a nontrivial,  pointed cone  with normlike-base $B_{K}$, and $A \subseteq E$ is a $\tau$-closed, nontrivial cone  with normlike-base $B_{A}$. Define
	$S_{-K} := {\rm conv}(-B_{K})$ and $S_{\partial A}^0 := {\rm conv}((B_{A} \cap {\rm bd}_\tau\,A) \cup \{0\})$ for normlike-bases $B_{K}$ and $B_{A}$. Consider the two assertions:	
	\begin{itemize}
		\item[$1^\circ$] 
        $
		({\rm cl}_\tau\, S_{\partial A}^0) \cap ({\rm cl}_\tau\, S_{-K}) = \emptyset.
		$
		\item[$2^\circ$]  There exists $(x^*, \alpha) \in K_\tau^{a\#} \cap (E^* \times \mathbb{P})$ such that
		\begin{align*}
		x^*(a) + \alpha \psi(a) & \geq  0 \geq x^*(k) + \alpha \psi(k)   \;\; \mbox{for all } a \in {\rm bd}_\tau\,A \mbox{ and }k \in -K, \\
		x^*(a) + \alpha \psi(a) &  >  0 > x^*(k) + \alpha \psi(k)   \;\; \mbox{for all } a \in ({\rm bd}_\tau\,A)\setminus \{0\} \mbox{ and }k \in -K \setminus \{0\}. 
		\end{align*}	
	\end{itemize}	
	If one of the sets ${\rm cl}_\tau\, S_{-K}$ and ${\rm cl}_\tau\, S_{\partial A}^0$ is $\tau$-compact, then $1^\circ \Longrightarrow 2^\circ$.\\
	If either $B_{K}$ is $\tau$-compact or ${\rm cl}_\tau\, S_{-K}$ is $\tau$-compact, $\psi$ is $\tau$-continuous (e.g., $\psi \in \mathcal{F}$), and $K$ is $\tau$-closed and convex, then $2^\circ \Longrightarrow 1^\circ$.
\end{corollary}

Let us consider the pure convex case (i.e., $K$ and $A$ are convex cones). 

\begin{theorem}   \label{cor:sep_main_top_iff_convex}
	Assume that $(E, \mathcal{F}, \tau)$ is a real locally convex space, $K, A \subseteq E$ are nontrivial, $\tau$-closed, convex cones with normlike-bases $B_{K}$ and $B_{A}$, and $K$ is pointed,  as well as $\psi$ is a $\tau$-continuous seminorm (e.g., $\psi \in \mathcal{F}$). Define $S_{-K} := {\rm conv}(-B_{K})$ and $S_{A}^0 := {\rm conv}(B_{A} \cup \{0\})$. 
	Suppose that ${\rm cl}_\tau\, S_{-K}$ is $\tau$-compact.	
	Then, the following assertions are equivalent:
	\begin{itemize}
		\item[$1^\circ$] $({\rm cl}_\tau\,S_{A}^0) \cap ({\rm cl}_\tau\, S_{-K}) = \emptyset$.		
		\item[$2^\circ$] There exists $(x^*, \alpha) \in K_\tau^{a\#} \cap (E^* \times \mathbb{P})$ such that \eqref{sep_non_top_1} and \eqref{sep_non_top_2} are valid.
		\item[$3^\circ$] $A \cap (-K \setminus \{0\}) = \emptyset$ and $0 \notin {\rm cl}_\tau\, S_{-K}$.
	\end{itemize}
\end{theorem}

\begin{proof}
	Theorem \ref{th:sep_main_top} provides $1^\circ \Longrightarrow 3^\circ$ while Theorem \ref{th:sep_main_top_iff} shows that $1^\circ \iff 2^\circ$. 
	The remaining implication $3^\circ \Longrightarrow 1^\circ$ follows by the fact $({\rm cl}_\tau\,S_{A}^0) \cap ({\rm cl}_\tau\, S_{-K}) \subseteq A \cap (-K \setminus \{0\})$ for $\tau$-closed, convex cones $A$ and $-K$, and $0 \notin {\rm cl}_\tau\, S_{-K}$.
\end{proof}

\begin{remark} 
    Assume that the assumptions of Theorem \ref{cor:sep_main_top_iff_convex} are satisfied and let $E$ be separated.
	According to Jahn \cite[Th. 3.22]{Jahn2011}, if the topology of $(E, \mathcal{F}, \tau)$ gives $E$ as the topological dual space of $E^*$, and ${\rm int}\, K_\tau^+ \neq \emptyset$, then each of the three assertions given in Theorem \ref{cor:sep_main_top_iff_convex} is equivalent to 
	\begin{itemize}
		\item[$4^\circ$] $0 \notin {\rm cl}_\tau\, S_{-K}$ and there exists $x^* \in E^*$ such that $
		x^*(a)   \geq  0 > x^*(k)$ for all $a \in A$ and  $k \in -K \setminus \{0\}$.
	\end{itemize}	
\end{remark}

\begin{remark}
	Theorem \ref{th:sep_main_top} shows that the condition $({\rm cl}_\tau\,S_{A}^0) \cap ({\rm cl}_\tau\, S_{-K}) = \emptyset$ implies $A \cap (-K \setminus \{0\}) = \emptyset$ and $0 \notin {\rm cl}_\tau\, S_{-K}$ (since $0 \in {\rm cl}_\tau\, S_{A}^0$) for nontrivial (not necessarily convex) cones $K$ and $A$ in  $E$.
	Without the convexity assumption, the conditions  $A \cap (-K \setminus \{0\}) = \emptyset$ and $0 \notin {\rm cl}_\tau\, S_{-K}$ do not imply the condition $({\rm cl}_\tau\,S_{A}^0) \cap ({\rm cl}_\tau\, S_{-K}) = \emptyset$, as to see in our Example \ref{ex:counterex_AcapcorK}.
\end{remark}

\section{Conclusions}

The separation of two sets (or more specific of two cones) plays an important role in different fields of mathematics (such as variational analysis, convex analysis, convex geometry, optimization). 
In our paper, we contributed with some new results for the separation of two (not necessarily convex) cones by a (convex) cone / conical surface in real (topological) linear spaces (see Section \ref{sec:cone_separation_lcs}). 
Basically, we followed the separation approach by Kasimbeyli \cite{Kasimbeyli2010} based on augmented dual cones and Bishop-Phelps type (normlinear) separating functions.  
As a key tool for deriving our main nonlinear cone separation theorems, we used classical separation theorems for convex sets.
Due to the generalization from real reflexive Banach spaces (as considered in \cite{Kasimbeyli2010}) to general real (topological) linear spaces and real locally convex spaces, respectively, we had to pay attention 
to some extended concepts (such as seminorms, seminorm-bases / normlike-bases of (not necessarily convex) cones, generalized interiority notions).

Concerning future research, we like to use our derived nonlinear cone separation theorems in order to develop separation theorems 
for (not necessarily convex) sets without cone properties in real (topological) linear spaces. 
Moreover, we aim to use our theorems for deriving some new scalarization results for general vector optimization problems in real (topological) linear spaces. 
In particular, an extension of the conic scalarization approach proposed by Kasimbeyli \cite{Kasimbeyli2010, Kasimbeyli2013} and corresponding applications for deriving duality statements (see, e.g., Kasimbeyli and Karimi \cite{KasimbeyliKarimi21}) are of interest.


\begin{thebibliography}{10}

\bibitem{AcaKas21}
Acar M, Kasimbeyli R (2021) A polyhedral conic functions based classification method for noisy data. J Ind Manag Optim 17:3493–3508
\bibitem{AdanNovo2004}
Adán M, Novo V (2004) Proper efficiency in vector optimization on real linear spaces. J Optim Theory Appl 124:515–540
\bibitem{AliprantisBorder}
Aliprantis CD, Border KC (2006) Infinite Dimensional Analysis: a Hitchhiker’s Guide. Springer, Berlin; London
\bibitem{BagAdi13}
Bagirov AM, Ugon J, Webb D, Ozturk G, Kasimbeyli R (2013) A novel piecewise linear classifier based on polyhedral conic and max-min separabilities. TOP 21:3–24
\bibitem{BP1962}
Bishop E, Phelps RR (1962) The support functionals of a convex set. Proc Symp Pure Math 7:27–35
\bibitem{BotGradWanka}
Boţ RI, Grad S-M, Wanka G (2009) Duality in Vector Optimization. Springer Berlin, Heidelberg
\bibitem{BouzaQuintanaTammer}
Bouza G, Quintana E, Tammer C (2019) A unified characterization of nonlinear scalarizing functionals in optimization. Vietnam J Math 47:683–713
\bibitem{BuiKru2018}
Bui HT, Kruger AY (2018) About extensions of the extremal principle. Vietnam J Math 46:215–242
\bibitem{BuiKru2019}
Bui HT, Kruger AY (2019) Extremality, stationarity and generalized separation of collections of sets. J Optim Theory Appl 182:211–264
\bibitem{MordukhovichEtAl2022}
Cuong DV, Mordukhovich BS, Nam NM, Sandine G (2023) Fenchel-Rockafellar theorem in infinite dimensions via generalized relative interiors. Optimization 72:135–162
\bibitem{DeumlichElsterNehse}
Deumlich R, Elster KH, Nehse R (1978) Recent results on separation of convex sets. Math Oper.forsch Stat 9:273–296
\bibitem{DrummondSvaiter}
Drummond LMG, Svaiter BF (2005) A steepest descent method for vector optimization. J Comput Appl Math 175:395–414
\bibitem{Durea2023}
Durea M (2023) Cone enlargements and applications to vector optimization. J Appl Numer Optim 5:55–70
\bibitem{Eichfelder2014}
Eichfelder G (2014) Variable Ordering Structures in Vector Optimization. Springer Berlin, Heidelberg
\bibitem{EichfelderKasimbeyli2014}
Eichfelder G, Kasimbeyli R (2014) Properly optimal elements in vector optimization with variable ordering structures. J Glob Optim 60:689–712
\bibitem{ElsterNehse1978}
Elster K, Nehse R (1978) Separation of two convex sets by operators. Comment Math Univ Carol 19:191–206
\bibitem{GasOzt06}
Gasimov RN, Ozturk G (2006) Separation via polyhedral conic functions. Optim Methods Softw 21:527–540
\bibitem{Gerstewitz1983}
Gerstewitz C (1983) Nichtkonvexe Dualität in der Vektoroptimierung. Wissensch Zeitschr TH Leuna-Merseburg 25:357–364
\bibitem{GerthWeidner}
Gerth C, Weidner P (1990) Nonconvex separation theorems and some applications in vector optimization. J Optim Theory Appl 67:297–320
\bibitem{GoeRiaTamZal2003}
Göpfert A, Riahi H, Tammer C, Zălinescu C (2003) Variational Methods in Partially Ordered Spaces. Springer New York, NY
\bibitem{Grad2015}
Grad SM (2015) Vector Optimization and Monotone Operators via Convex Duality. Springer Cham
\bibitem{Khazayel2021b}
Günther C, Khazayel B, Tammer C (2022) Vector optimization w.r.t. relatively solid convex cones in real linear spaces. J Optim Theory Appl 193:408–442
\bibitem{Khazayel2022}
Günther C, Khazayel B, Tammer C (2023) Duality assertions in vector optimization w.r.t. relatively solid convex cones in real linear spaces. Minimax Theory its Appl (to appear)
\bibitem{GutHuergaNovo2018}
Gutiérrez C, Huerga L, Vicente Novo (2018) Nonlinear scalarization in multiobjective optimization with a polyhedral ordering cone. Int Trans Oper Res 25:763–779
\bibitem{Ha2022}
Ha TXD (2023) A unified scheme for scalarization in set optimization. Optimization. https://doi.org/10.1080/02331934.2023.2171730
\bibitem{HaJahn2017}
Ha TXD, Jahn J (2017) Properties of Bishop-Phelps cones. J Nonlinear Convex Anal 18:415–429
\bibitem{HaJahn2021}
Ha TXD, Jahn J (2021) Bishop–Phelps cones given by an equation in Banach spaces. Optimization 72(5):1309-1346
\bibitem{Henig}
Henig MI (1982a) A cone separation theorem. J Optim Theory Appl 36:451–455
\bibitem{Henig2}
Henig MI (1982b) Proper efficiency with respect to cones. J Optim Theory Appl 36:387–407
\bibitem{HiriartUrruty}
Hiriart-Urruty JB (1979) New concepts in nondifferentiable programming. Mémoires de la Société Mathématique de France 60:57–85
\bibitem{Holmes}
Holmes RB (1975) Geometric Functional Analysis and Its Applications. Springer New York, NY
\bibitem{HuergaJadambaSama2019}
Huerga L, Jadamba B, Sama M (2019a) An extension of the Kaliszewski cone to non-polyhedral pointed cones in infinite-dimensional spaces. J Optim Theory Appl 181:437–455
\bibitem{HuergaKhanSama2019}
Huerga L, Khan AA, Sama M (2019b) A {H}enig conical regularization approach for circumventing the {S}later conundrum in linearly $\ell_{+}^{p}$-constrained least squares problems. J Appl Numer Optim 1:117–129
\bibitem{Jahn2009}
Jahn J (2009) Bishop–Phelps cones in optimization. Int J Optim: Theory Methods Appl 1:123–139
\bibitem{Jahn2011}
Jahn J (2011) Vector Optimization: Theory, Applications, and Extensions. Springer Berlin, Heidelberg
\bibitem{Jahn2022}
Jahn J (2022) Characterizations of the set less order relation in nonconvex set optimization. J Optim Theory Appl 193:523–544
\bibitem{Jahn2023}
Jahn J (2023) A unified approach to Bishop-Phelps and scalarizing functionals. J Appl Numer Optim 5:5–25
\bibitem{NKasimbeyli15}
Kasimbeyli N (2015) Existence and characterization theorems in nonconvex vector optimization. J Global Optim 62:155–165
\bibitem{NKasimbeyli19}
Kasimbeyli N (2019) Separation theorem via superlinear functions and characterization of maximal elements in multiobjective optimization. Appl Anal Optim 3:373–382
\bibitem{KasKas17}
Kasimbeyli N, Kasimbeyli R (2017) A representation theorem for Bishop-Phelps cones. Pac J Optim 13:55–74
\bibitem{Kasimbeyli2010}
Kasimbeyli R (2010) A nonlinear cone separation theorem and scalarization in nonconvex vector optimization. SIAM J Optim 20:1591–1619
\bibitem{Kasimbeyli2013}
Kasimbeyli R (2013) A conic scalarization method in multi-objective optimization. J Glob Optim 56:279–297
\bibitem{KasimbeyliKarimi2019}
Kasimbeyli R, Karimi M (2019) Separation theorems for nonconvex sets and application in optimization. Oper Res Lett 47:569–573
\bibitem{KasimbeyliKarimi21}
Kasimbeyli R, Karimi M (2021) Duality in nonconvex vector optimization. J Global Optim 80:139–160
\bibitem{Kasimbeyli2019}
Kasimbeyli R, Ozturk ZK, Kasimbeyli N, Yalcin GD, Erdem B (2019) Comparison of some scalarization methods in multiobjective optimization. Bull Malays Math Sci Soc 42:1875–1905
\bibitem{KhanSama2013}
Khan AA, Miguel Sama (2013) A new conical regularization for some optimization and optimal control problems: Convergence analysis and finite element discretization. Numer Funct Anal Optim 34:861–895
\bibitem{Khanetal2015}
Khan AA, Tammer C, Zălinescu C (2015) Set-valued Optimization: An Introduction with Applications. Springer Berlin, Heidelberg
\bibitem{Khazayel2021a}
Khazayel B, Farajzadeh A, Günther C, Tammer C (2021) On the intrinsic core of convex cones in real linear spaces. SIAM J Optim 31:1276–1298
\bibitem{KruMor80a}
Kruger AY, Morduhovič BŠ (1980) Extremal points and the Euler equation in nonsmooth optimization problems. Dokl Akad Nauk BSSR 24
\bibitem{Kru98}
Kruger AY (1998) On the extremality of set systems. Dokl Nats Akad Nauk Belarusi 42:24–28, 123
\bibitem{Roshchina2021}
Millán RD, Roshchina V (2023) The intrinsic core and minimal faces of convex sets in general vector spaces. Set-Valued Var Anal 31.  https://doi.org/10.1007/s11228-023-00671-6
\bibitem{Mor76}
Mordukhovich BS (1976) Maximum principle in the problem of time optimal response with nonsmooth constraints. Prikl Mat Meh 40:960–969
\bibitem{Mor94}
Mordukhovich BS (1994) Generalized differential calculus for nonsmooth and set-valued mappings. J Math Anal Appl 183:250–288
\bibitem{Mordukhovich2006a}
Mordukhovich BS (2006) Variational Analysis and Generalized Differentiation I: Basic Theory. Springer Berlin, Heidelberg
\bibitem{MorNam2022}
Mordukhovich BS, Nam NM (2022) Convex Analysis and Beyond, Volume I: Basic Theory. Springer Cham
\bibitem{Nehse1981}
Nehse R (1981) A new concept of separation. Comment Math Univ Carol 22:169–179
\bibitem{Nieuwenhuis}
Nieuwenhuis JW (1983) About separation by cones. J Optim Theory Appl 41:473–479
\bibitem{NovoZali2021}
Novo V, Zălinescu C (2021) On relatively solid convex cones in real linear spaces. J Optim Theory Appl 188:277–290
\bibitem{OztGur15}
Ozturk G, Bagirov AM, Kasimbeyli R (2015) An incremental piecewise linear classifier based on polyhedral conic separation. Mach Learn 101:397–413
\bibitem{Soltan2021}
Soltan V (2021) Separating hyperplanes of convex sets. J Convex Anal 28:1015–1032
\bibitem{Swartz}
Swartz C (1992) An Introduction to Functional Analysis. CRC Press
\bibitem{TammerWeidner2020}
Tammer C, Weidner P (2020) Scalarization and Separation by Translation Invariant Functions. Springer Cham
\bibitem{Zaffaroni2003}
Zaffaroni A (2003) Degrees of efficiency and degrees of minimality. SIAM J Control Optim 42:1071–1086
\bibitem{Zaffaroni2022}
Zaffaroni A (2022) Separation, convexity and polarity in the space of normlinear functions. Optimization 71:1213–1248
\bibitem{Zali2002}
Zălinescu C (2002) Convex Analysis in General Vector Spaces. World Scientific, River Edge

	
	
\end{thebibliography}
\end{document}